\documentclass[a4paper,12pt]{opelsarticle}
\usepackage{amsmath,amsfonts,graphicx,tikz,pgfplots, tikz-3dplot,bm, pdfsync}

\usepackage[colorlinks=true, citecolor=red]{hyperref}
\usepackage{fullpage}
\usetikzlibrary{patterns}

\newcommand{\ds}{\displaystyle}
\newtheorem{remark}{Remark}
\newtheorem{theorem}{Theorem}
\newtheorem{lemma}{Lemma}

\newtheorem{Cor}{Corollary}
\newtheorem{remarks}{Remarks}
\def \e {{\rm e}}
\def \R{{\mathbb{R}}}
\def \S{{\mathbb{S}}}

\newcommand{\be}{\begin{equation}}
\newcommand{\ee}{\end{equation}}

\def\p{\partial}
\def\n{\nabla}
\def\nablav{\bm\nabla}
\def\d{\hbox{d}}

\def\eq#1{\begin{equation}#1\end{equation}}
\def\eeq#1{\begin{eqnarray}#1\end{eqnarray}}
\def\eeqn#1{\begin{eqnarray*}#1\end{eqnarray*}}

\def\vx{{\bf x}}
\def\vu{{\bf u}}
\def\omegav{{\bm\omega}}
\def\vn{{\bf n}}

\begin{document}

\title{Stratified Radiative Transfer for Multidimensional Fluids\footnote{Will be submitted for publication in Compte-Rendus de M\'ecanique}}
\author{
Fran\c cois Golse\footnote{\emph{francois.golse@polytechnique.edu}, CMLS, Ecole polytechnique, 91128 Palaiseau Cedex, France},
Olivier Pironneau\footnote{\emph{olivier.pironneau@sorbonne-universite.fr }, LJLL, Sorbonne Universit\'e, Paris, France.}
}

\parindent=0pt
\begin{frontmatter}
\begin{abstract}
New mathematical and numerical results are given for the coupling of the temperature equation of a  fluid with  Radiative Transfer: existence and uniqueness and a convergent monotone numerical scheme.  The technique is shown to be feasible  for studying the temperature of lake Leman  heated by the sun and for the earth atmosphere to study the effects of greenhouse gases.
 \end{abstract}
\end{frontmatter}


\section{Introduction}

Fifty years ago, the second author was admitted to the prestigious Dept of Applied Math. \& Theoretical Physics at Cambridge, UK, headed then by Sir James Lighthill.
Two \texttt{ibm} card punchers connected to the computing center -- also one of the best in the world in those days-- had been relegated to the basement; to use them was frowned upon as a threat to the speciality of the lab: clever analytic approximations and other multiple scales expansions of special cases of the  Navier-Stokes equations.

It took a decade to prove that computer simulations for fluids were not only possible, but also useful to industry.  A colleague from the wind tunnels in Modane told us then that an airplane could never be designed and validated by a numerical simulation. True to this wrong prediction however, many ad-hoc turbulence models had to be devised: it was only by a combined theoretical, experimental and computational (TEC) effort that the world's first complete airplane could be simulated at Dassault Aviation in 1979 and that airplanes have since be flown safely without the difficult certification stamps of wind tunnels.

It was also a success of the top-down approach to CFD. The ``JLL''(Lions) school of applied mathematics had the luck of being taken seriously by a few French high-tech industry labs. This was not the case in the USA where the head of a national research funding agency  had ruled out variational methods (leading to finite volumes and finite elements for fluids) as ``incomprehensible by aeronautical engineers", thereafter forcing all numerical schemes to be in the class of body fitted structured meshes, an impossible task for airplanes.

The top-down approach to a problem could be defined by saying that the mathematical model is defined first, then shown to be well posed and then approximated numerically by convergent algorithms.  The bottom-up approach is when the problem is made of several modules, studied independently, and patched together  at the algorithmic level. 

The downside of the top-down approach - from functional analysis to numerical methods - is that it may  discard important faster algorithms for which convergence are not known. This was the case for compressible flows in the nineties for which the bottom-up approach  pragmatically patched different turbulence and/or numerical models in different zones with the drawback that it was difficult to assert that the computed solution was one of the original problem. 

In the numerical simulations which fill the supercomputing centers today, CFD is often only one part of a multi-physics model. Such are the combustion and climate computations.  Both need, at least, radiative transfer and chemistry modules.

While the top-down approach is successful in computational chemistry \cite{CANCES20033}, mathematical analysis of climate models is still in progress. The  three dimensional Primitive Equations with hydrostatic and geostrophic approximations have been shown to be well posed  (see \cite{LionsTemamWang},\cite{AZE},\cite{TIT} and the bibliography therein) and so are the multi-layered Shallow Water equations for the oceans \cite{CAS}; but even if the coupled ocean-atmosphere is mathematically well-posed, it is very far from the complete model used in climatology. No doubt when a new numerical climate project is proposed, such as \cite{DYN}, a top-down approach is made \cite{DUB}, but soon overwhelmed by the complexity of the task when more modules are added.

Radiative transfer -- one such module that needs to be added -- is essential in astrophysics \cite{CHA} to derive the composition of stars, in nuclear engineering to predict plasma\cite{DAU}, in combustion for engines \cite{AMA}, and many other fields like solar panels \cite{Zengeabi5484} and  even T-shirts \cite{Zengeabi5484}!

In the eighties, at CEA, R. Dautray \cite{DAU} headed a team of applied mathematicians  who used the top-down approach in nuclear engineering. The first author was in close contact with them.  But turning his expertise on radiative transfer to climate modeling is not straightforward.

Books on radiative transfer for the atmosphere are numerous, such as \cite{GOO}, \cite{BOH}and \cite{ZDU}; but to speed-up codes, the documentation manual of climate models reveal that many approximations are made.
For instance LMDZ refers to  a model proposed  by Fouquart \cite{FOU}\cite{MOR} which suggests that empirical formulas are used in addition to simplified numerical schemes to speed-up the computations.  The formulas for the absorption, scattering and albedo coefficients are complex and adapted to reproduce the experimental data.
In other words the gap is wide between practice and  fundamentals as seen by Fowler \cite{FOW} and Chandrasekhar \cite{CHA}, for instance.

Coupling radiative transfer to the Navier-Stokes system using the top-down approach is the topic of this article. The problem is shown well posed in the context of a stratified atmosphere and a numerical method -- derived from the mathematical proof of  well posedness -- is proposed. It is accurate in the sense that there are no singular functions or integrals to approximate. It is fast compared to the fluid solver to which it is coupled but of course not as fast as empirical formulas.

\section{Radiative transfer and the temperature equation}

Let us begin with a simple problem: the effect of sunlight on a lake $\Omega$. 
Let $I_\nu(\vx,\omega,t)$ be the light intensity of frequency $\nu$ at $\vx\in\Omega$, in the direction $\omega\in \S^2$, the unit sphere, at time $t\in(0,T)$.
 Let $T,\rho,\vu$ be the temperature,  density and  velocity in the lake. Energy , momentum and  mass conservations (see \cite{POM},\cite{FOW}) yields (\ref{onea}),(\ref{oneb}),(\ref{onec}):

\subsection{The fundamental equations}~

Given $I_\nu,T$ at time zero, find $I_\nu,T$ for all $\{\vx,\omegav,t,\nu\}\in\Omega\times\S_2\times(0,T)\times\R^+$ such that
\eeq{\label{onea}&&
\frac1c\p_t I_\nu + \omegav\cdot\nablav I_\nu+\rho\bar\kappa_\nu a_\nu\left[I_\nu-{\frac{1}{4\pi}\int_{\S^2}} p(\omegav,\omegav')I_\nu(\omegav')\d\omega'\right]
= \rho\bar\kappa_\nu(1-a_\nu) [B_\nu(T)-I_\nu],
\\ \label{oneb}& &
\p_tT+\vu\cdot\n T -\kappa_T\Delta T=-\nablav\cdot\int_0^\infty{\int_{\S^2}} I_\nu(\omegav')\omegav \d\omega \d\nu. 
\\ \label{onec}&&
\p_t\vu+\vu\cdot\n\vu-\frac{\mu_F}\rho\Delta\vu + \frac1\rho\n p={\bf g},\quad \n\cdot\vu=0,
\quad 
\p_t\rho+\n\cdot(\rho\vu)=0,
}
where $\nabla,\Delta$ are with respect to $\vx$, 
$\ds B_\nu(T)=\frac{2 \hbar \nu^3}{c^2[{\rm e}^\frac{\hbar\nu}{k T}-1]}$, is the Planck function,
  $\hbar$ is the Planck constant, $c$ is the speed of light in the medium and $k$ is 
the Boltzmann constant.
The absorption coefficient $\kappa_\nu:=\rho\bar\kappa_\nu$ is  the percentage of light absorbed per unit length, $a_\nu\in(0,1)$ is  the scattering albedo, $\frac1{4\pi}p(\omegav,\omegav')$ is the probability that a ray in the direction $\omegav'$ scatters in the direction $\omegav$. The constants $\kappa_T$ and $\mu_F$ are the thermal and molecular diffusions; ${\bf g}$ is the gravity. 

Existence of solution for (\ref{onec}) has been established by P-L. Lions \cite{PLL}.


As $c>>1$, in a regime where  $\frac1c\p_t I_\nu<<1$,  integrating (\ref{onea}) in $\bm\omega$  leads to an alternative form for (\ref{oneb}):
\eq{
\p_tT+\vu\cdot\n T -\kappa_T\Delta T
= - \int_0^\infty \rho \bar\kappa_\nu(1-a_\nu)\left(4\pi B_\nu(T)-\int_{\S^2} I_\nu(\omegav )\d\omega\right) \d\nu.
}
  As usual, boundary conditions  must be given. Dirichlet or Neumann conditions may be prescribed for $\vu$ and $T$ on $\p\Omega $.  For the light intensity equation, $I_\nu$ should be given  at all times on $\{ (\vx, \omegav) \in \p \Omega\times \S^2:\quad \vn(\vx)\cdot \omegav <0\}$, 
where $\vn$ is the outer unit normal of $\p\Omega$. Finally $\rho$ should be specified on on $\p\Omega $ when $\vu\cdot\vn<0$.

\subsection{Grey Medium} 
 When $\kappa_\nu$ and $a_\nu$ are independent of $\nu$ - a so-called {\it grey medium} (cf. \cite{FOW}, p. 70)- the problem  can be written in terms of $I=\int_0^\infty I_\nu\d\nu$:
\eeq{\label{oneam}&&
 \omegav\cdot\nablav I+\kappa a\left[I-{\frac{1}{4\pi}\int_{\S^2}} p(\omegav,\omegav')I(\omegav')\d\omega'\right]
= \kappa(1-a) (B_0T^4-I),
\\ \label{onebm}& &
\p_tT+\vu\cdot\n T -\kappa_T\Delta T=  - \kappa(1-a)4\pi\left( B_0 T^4-\frac1{4\pi}\int_{\S^2} I(\omegav )\d\omega\right) ,
}
where $B_0$ comes from the Boltzmann-Stefan law:
\[
\int_0^\infty \frac{2 \hbar \nu^3}{c^2[{\rm e}^\frac{\hbar\nu}{k T}-1]}\d\nu
=
\left(\frac{\hbar}{k T}\right)^{-4}\frac{2\hbar}{c^2}\int_0^\infty \frac{ \left(\frac{\hbar\nu}{k T}\right)^3}{{\rm e}^\frac{\hbar\nu}{k T}-1}\d\frac{\hbar\nu}{k T}
=B_0 T^4 \hbox{ with } B_0:=\frac{2 k^4}{\hbar^3c^2}\frac{\pi^4}{15}.
\]
~ 

\subsection{Vertically stratified cases: spatial invariance}~

Let $(x,y,z)$ be a cartesian frame with $z$ the altitude/depth.
The sun being very far, the light source on the lake is independent of $x$ and $y$. Then, assuming that $T'$ varies slowly with $x$ and $y$, in the sense that
\eeq{&&\label{slow}
 (H)~~~ \left.\hskip0.1\linewidth\begin{matrix}
\p_z I_\nu >> \p_x I_\nu,~\quad \p_z I_\nu >>\p_y I_\nu,
\cr 
\end{matrix}\right.
}
then (\ref{onea}),(\ref{oneb}) become \cite{ZDU}
 \eeq{&& \label{oneamu}
\mu\p_z I_\nu + \kappa_\nu I_\nu 
= \kappa_\nu(1-a_\nu) B_\nu(T)+\frac{\kappa_\nu a_\nu}2\int_{-1}^1 p(\mu,\mu')I_\nu(z,\mu')\d\mu'
\\ && \label{oneabmu}
~I_\nu(z_M,\mu)|_{\mu<0}= Q^-(\mu) B_\nu({\bar T_S}), ~I(z_m,\mu)|_{\mu>0}=0,
\\&&\label{onebmu}
\p_t T+\vu\cdot\n T  - \kappa_T\Delta T
=-4\pi\int_0^\infty  \kappa_\nu(1-a_\nu)\left(B_\nu(T)-\tfrac12\int_{-1}^1 I_\nu\d\mu\right) \d\nu,\quad \p_n T|_{\p\Omega}=0.
\cr&&
}
where $z_M(x,y)$ and $z_m(x,y)$ are max and min of z such that $(x,y,z)\in \Omega$, $\mu$ is the cosine of the angle ${\bm\omega}$ to the vertical axis, $Q^-(\mu)=-\mu Q'\cos\theta$ is the sunlight intensity when $\theta$ is the latitude, and ${\bar T_S}$ is the temperature of the sun; we have assumed that the sun is a black body and that no light comes back from the bottom of the lake. Here $\vu$ is given, solenoidal and regular enough for (\ref{onebmu}) to make sense.

\begin{remarks}
~

\begin{itemize}
\item Hypothesis (H) will hold if $T$ varies slowly with $x,y$. It will be so if $\vu$ is almost horizontal and the vertical cross sections of $\Omega$ depend slowly on $x,y$. Turbulent flows do not satisfy this criteria. 

\item According to our definition of top-down analysis, the problem investigated is  (\ref{oneamu}),(\ref{oneabmu}),(\ref{onebmu}), not (\ref{onea}),(\ref{oneb}),(\ref{onec}), justifying the restriction ``stratified'' in the title.
\item All terms of (\ref{onebmu}) must be kept, except maybe, $\kappa_T\p_{xx}T$ and $\kappa_T\p_{yy}T$, but neglecting them renders the boundary conditions mathematically difficult.
\item We shall ignore the mathematical difficulty induced by the boundary condition $\p_n T|_{\p\Omega}=0$ when the intersection of the side of the lake with the water surface is not at right angle.
\end{itemize}
\end{remarks}
~
\subsection{The vertically stratified grey problem}~

For a grey medium (\ref{oneamu}),(\ref{onebmu}) become
 \eeq{ \label{oneamugs}
(P^1)\left\{\begin{matrix}\ds \mu\p_z I + \kappa I 
= \kappa(1-a) B_0 T^4+\frac{\kappa a}2\int_{-1}^1 p I\d\mu'
,~I|_{z_M,\mu<0}=-\mu Q B_0 {\bar T_S}^4, ~I|_{z_m,\mu>0}=0,
\cr
\ds \p_t T+\vu\cdot\n T -\kappa_T\Delta T
=-4\pi\kappa(1-a)\left(B_0 T^4-\tfrac12\int_{-1}^1 I\d\mu\right),~~ \p_n T|_{\p\Omega}=0.
\end{matrix}
\right.}

\subsection{Elimination of $I$ when the scattering is isotropic}~

Denote the exponential integral and the mean light intensity respectively by 
\[
\ds E_m(x):=\int_0^1\mu^{m-2}{\rm e}^{-\frac x\mu}\d\mu, \qquad \ds J(z):=\tfrac12\int_{-1}^1 I(z,\mu)\d\mu.
\]  
Then the method of characteristics applied to (\ref{oneamugs})  gives 
\eq{
(P^2)\left\{\begin{matrix}\ds
J(z) = \tfrac12 Q B_0 {\bar T_S}^4E_3(\kappa(z_M-{z}))
+ \ds \tfrac12\int_{z_m}^{z_M} \kappa E_1(\kappa|s-z|)\left((1-a)B_0 T_s^4+a J(s)\right)\d s,
\cr 
\ds \p_t T+\vu\cdot\n T- \kappa_T\Delta T
=-4\pi\kappa(1-a)\left(B_0 T^4(z)- J(z)\right)
\end{matrix}\right.
}
Note that to improve readability, we write indifferently $T(z)$ or $T_z$.

\subsection{No scattering}~

Let $T_e(z)=\left(\tfrac12 Q E_3(\kappa|z_M-{z}|)\right)^\frac14 {\bar T_S}$ and assume that $a=0$, then
{\small\eq{\label{st4}
(P^3)\left\{~~~\begin{matrix}\ds
(4\pi\kappa B_0)^{-1}(\ds \p_t T+\vu\cdot\n T-\kappa_T\Delta T) + T^4
= T_e^4 + \tfrac12\int_{z_m}^{z_M} \kappa E_1(\kappa|s-{z}|)T_s^4\d s, \quad \p_n T|_{\p\Omega}=0.
\end{matrix}\right.
}}

\subsection{Algorithm for $(P^3)$ in the stationary static case}
Assume  $T$ stationary and $\vu=0$. Let $\bar\kappa_T=(4\pi\kappa B_0)^{-1}\kappa_T$. 

Generate $\{T^n\}_{n\geq 0}$ from $T^0=0$ by, 
\eeq{\label{algo}&&
\left|~~~\begin{matrix}\ds
(T^{n+\tfrac12})^4 := T_e^4 + \tfrac12\int_{z_m}^{z_M} \kappa E_1(\kappa|s-{z}|){T_s^n}^4\d s,\quad T^{n+\tfrac12}\ge 0
\cr
-\bar\kappa_T\Delta T^{n+1} + (T^{n+1}_+)^4 = (T^{n+\tfrac12})^4, \quad \p_n T^{n+1}|_{\p\Omega}=0.
\end{matrix}\right.
}
where $T_+=\max(T,0)$.
Note that $T\mapsto-\bar\kappa_T \Delta T+T_+^4$ is a monotone operator for which Newton or fixed point iterations can be applied to solve the PDE.  To prove monotone convergence, the following result is needed.
\begin{lemma} 
$
C_1(\kappa) :=\tfrac12 \max_z\int_0^Z \kappa E_1(\kappa|s-z])\d s < 1.
$
\end{lemma}
\begin{proof}
\eeq{&&
\int_0^X E_1(x)\d x=\int_1^\infty\int_0^X \frac{\e^{-x t}}{t}\d x\d t = 
\int_1^\infty \frac{1-\e^{-X t}}{t^2}\d t < \int_1^\infty \frac{1}{t^2}\d t=1.
\cr&& \Rightarrow ~~
\kappa\int_0^Z E_1(\kappa|\tau-t|)\d t = \int_0^{\kappa Z} E_1(|s-\kappa\tau|)\d s=
\int_0^{\kappa\tau} E_1(\kappa\tau-s)\d s +
\int_{\kappa\tau}^{\kappa Z} E_1(s-\kappa\tau)\d s
\cr&&
=
 \int_0^{\kappa\tau} E_1(\theta)\d\theta + \int_0^{\kappa (Z-\tau)} E_1(\theta)\d\theta  < 2.
}
\end{proof}
\begin{theorem}\label{th:one} ~

$\{T^n\}_{n\geq 0}$ generated by Algorithm (\ref{algo}) converges to a solution of (\ref{st4}) and the convergence is monotone: $T^{n+1}(\vx)>T^n(\vx)$ for all $\vx$ and all $n$.
\end{theorem}
\begin{proof}:~
From (\ref{algo})
\[
(T^{n+\tfrac12})^4 \leq |T_e^4|_\infty +  C_1(\kappa)|{T^n}|^4_\infty.
\]
By the maximum principle for the PDE in (\ref{algo}), $T^{n+1}\geq 0$ and $|T^{n+1}|_\infty \leq |T^{n+\tfrac12}|_\infty$ , therefore
\[
|T^{n+1}|^4\leq |T_e|^4_\infty +  C_1(\kappa)|{T^n}|^4_\infty.
\]
Hence $|T^{n+1}|_\infty$ is bounded. Assume that $T^n\ge T^{n-1}$.  The convergence is monotone because
\[
(T^{n+\tfrac12})^4-(T^{n-\tfrac12})^4 = \tfrac12\int_{z_m}^{z_M} \kappa E_1(\kappa|s-{z}|)\left[(T_s^{n})^4-(T_s^{n-1})^4\right] \ge 0,
\]
and as
\eeq{&&
-\bar\kappa_T\Delta(T^{n+1}-T^n) + b(T^{n+1}-T^n)=(T^{n+\tfrac12})^4-(T^{n-\tfrac12})^4
}
with $b=((T^{n+1})^2+(T^n)^2)(T^{n+1} +T^n)\geq 0$, the maximum principle implies that $T^{n+1}-T^n\ge 0$.

\end{proof}
\begin{remark}
Generalization of the above result to $(P^3)$ is straightforward because the maximum principle holds also for the temperature equation with convection. Consequently it seems doable to extend the above to the system (\ref{oneb}),(\ref{onec}).
When the density variations with the temperature are small the Boussineq approximation can be used in conjunction with (\ref{st4}):
\eq{
(P^4)\left\{~~~\begin{matrix}\ds
(4\pi\kappa B_0)^{-1}(\ds \p_t T+\vu\cdot\n T-\kappa_T\Delta T) + T^4
= T_e^4 + \tfrac12\int_{z_m}^{z_M} \kappa E_1(\kappa|t-{z}|)T_s^4\d s,
\cr
\p_t\vu+\vu\cdot\n\vu-\nu_F\Delta\vu + \n p=-b(T-T_0){\bf g},\quad \n\cdot\vu=0,
\end{matrix}\right.
}
with $\vu,T$ given at $t=0$ and $\vu$ or $\p_n\vu$ or $p\vn + \nu_T\p_\vn\vu$ and $\p_nT=0$ or $T$ given on $\p\Omega$. The kinematic viscosity  $\nu_F=\mu_F/\rho$ is taken constant; $b$ is a measure of $\p_T\rho$ and $T_0$ is the average temperature.  See \cite{ATT}, for instance, for the mathematical analysis of the Boussinesq-Stefan problem (similar to $(P^4)$ without the $T^4$ terms).
 \end{remark}

\section{Numerical tests}
\begin{table}[htp]
\caption{{\it The physical constants.}}
\begin{center}
\begin{tabular}{|c|c|c|c|}
\hline
$c$ & $\hbar$ & $k$& $B_0$\cr
\hline
$2.998\times 10^8 $& $6.6261\times 10^{-34}$ & $1.381\times 10^{-23}$ & $1.806657\times 10^{-19}$ \cr
\hline
\end{tabular}
\end{center}
\label{cte}
\end{table}%
Earth sees the sun as a black body at temperature ${\bar T_S}=5800$K radiating with an intensity $Q'=1370W/m^2$ of which 70\% reach the ground, giving at noon in Milano $Q=1370\times 0.7\cos\frac\pi 4= 678$. 

For water $\rho=1000kg/m^3$; light absorption is $\kappa=0.1$ for one meter and thermal diffusivity of water is $\kappa_T =1.5\times 10^{-7}m^2/s$ giving $\bar\kappa_T = 0.66\times 10^{11}$.

To avoid those large numbers we scale $T$ by $10^{-3}$.  Then ${\bar T_S}=5.8$, $(\frac Q 2)^\frac14 {\bar T_S}= 24.9$, $\tilde\kappa_T=10^{-9}\bar\kappa_T=66$.

\subsection{A one dimensional test}
If $\Omega=(0,10)$, we need to solve with Algorithm (\ref{algo}) the integro-differential equation in $z$:
\eeq{&&\label{oneD}
-66 T^{\prime\prime} + T^4 = 12.5 E_3(0.1|10-z|) + 0.05\int_0^{10} E_1(0.1|s-z|)T^4(s)\d s,\quad
\cr&&
T(0)=(12.5E_3(0))^\frac14,\; T^\prime(10)=0.
}
To solve $-66 T^{\prime\prime} + T^4 =f$, 3 iterations of a fixed point loop are used: 
 $-830 T^{\prime\prime m+1} + {T^{m}}^3T^{m+1} =f$.
 
The results are shown on Figure \ref{oneDlac}. The convergence is monotone as expected, even though Theorem \ref{th:one} hasn't been proved when a Dirichlet condition is applied to $T$ on part of $\p\Omega$. Notice that in absence of sunlight the temperature would be $T(0)$ everywhere.
 \begin{figure}[tbp]
\begin{minipage} [b]{0.4\textwidth} 
\begin{center}
\begin{tikzpicture}[scale=0.7]
\begin{axis}[legend style={at={(0,1)},anchor=north west}, compat=1.3,
  xlabel= {$z$},
  ylabel= {scaled $T^n$}]
\addplot[thick,solid,color=black,mark='+',mark size=1pt] table [x index=0, y index=1]{tempe.txt};
\addlegendentry{ $T^4_z$}
\addplot[thick,solid,color=blue,mark='+',mark size=1pt] table [x index=0, y index=2]{tempe.txt};
\addlegendentry{ $T^5_z$}
\addplot[thick,solid,color=red,mark='+',mark size=1pt] table [x index=0, y index=3]{tempe.txt};
\addlegendentry{ $T^6_z$}
\addplot[thick,solid,color=yellow,mark='+',mark size=1pt] table [x index=0, y index=4]{tempe.txt};
\addlegendentry{ $T^7_z$}
\addplot[thick,solid,color=green,mark='+',mark size=1pt] table [x index=0, y index=5]{tempe.txt};
\addlegendentry{ $T^8_z$}
\addplot[thick,solid,color=pink,mark='+',mark size=1pt] table [x index=0, y index=6]{tempe.txt};
\addlegendentry{ $T^9_z$}
\end{axis}
\end{tikzpicture}
\caption{
Convergence of $T^n$ solution of (\ref{oneD})
}
\label{oneDlac}
\end{center}
\end{minipage}
\begin{minipage} [b]{0.5\textwidth} 
\begin{center}
\includegraphics[width=8cm]{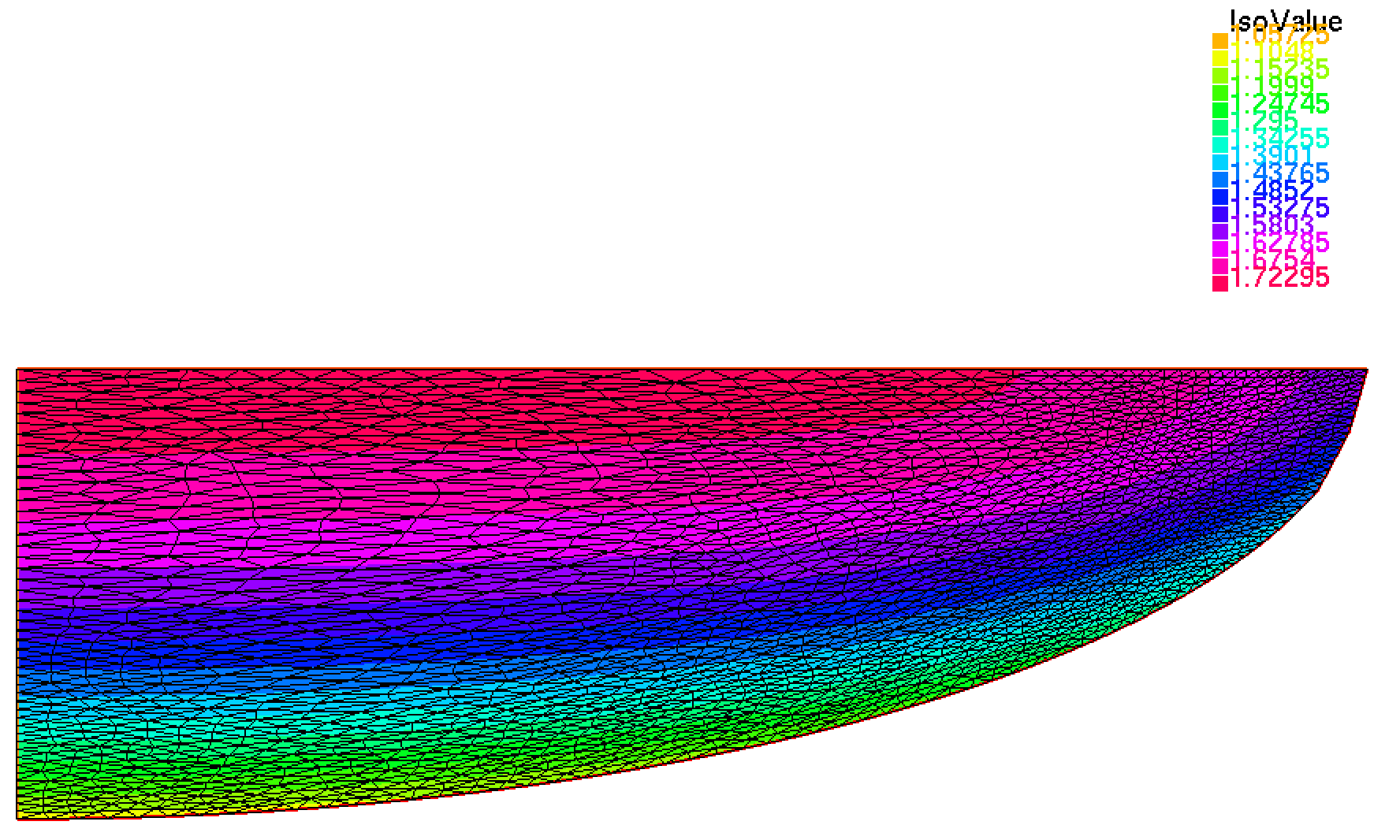}
\caption{Color map of $T(x,z)$ at iteration 10. The triangulation is also shown, adapted  from  $T$ computed at iteration 5.}
\label{twoDlac}
\end{center}
\end{minipage}
\end{figure}
\subsection{A two dimensional test for a lake}~

Now $\Omega$ is half of the vertical cross section of a symmetric lake. The lower right quarter side of the unit circle is stretched by $x,z\mapsto 30 x,10z$. The bottom boundary has an equation named $z=z_m(x)$.
The same problem is solve in 2D:
\eeq{\label{twoD}&&
-66 \Delta T + T^4 = 12.5 E_3(0.1|z_m(x)-z|) + 0.05\int_{z_m(x)}^{10} E_1(0.1|s-z|)T^4(s)\d s,\cr&&
T(x,z_m(x))=(12.5E_3(0.1 z_m(x))^\frac14,\; \p_z T(10)=0.
}
The same $3\times 10$ double iteration loop is used ;  the results are shown on Figure \ref{twoDlac}.

\subsection{A 3D case with convection in Lake Leman}
Lake Leman is discretized into 33810 tetrahedra. The surface  has 1287 triangles. The Finite Element method of degree 1 is used. This is too coarse for a Navier-Stokes simulation but appropriate for a potential flow. Pressure is imposed on the left and right tips  to simulates the debit of the Rh\^one.
The pressure $p$ solves $-\Delta p=0$ with $\p_np=0$ on the remaining boundaries; the velocity is $\vu=\n p$. The top plot in Figure \ref{leman} shows $p$ and $\vu$.

The full temperature equation of Problem ($P^4$) is solved with the same physical constant as above. The temperature is set at $T_e$ initially and on the bottom and side boundaries of the lake.
The time step is $t=0.1$; the method is fully implicit for the temperature. At each time step 3 iterations are needed to handle the $T^4$ term.  Figure \ref{leman} shows the temperature after 15 time steps; it appears to have reached a steady state.    The top right view of Figure \ref{leman} shows a region in red where the water at the surface is the hotest.
\begin{figure}[htbp]
\begin{center}
\includegraphics[width=6cm]{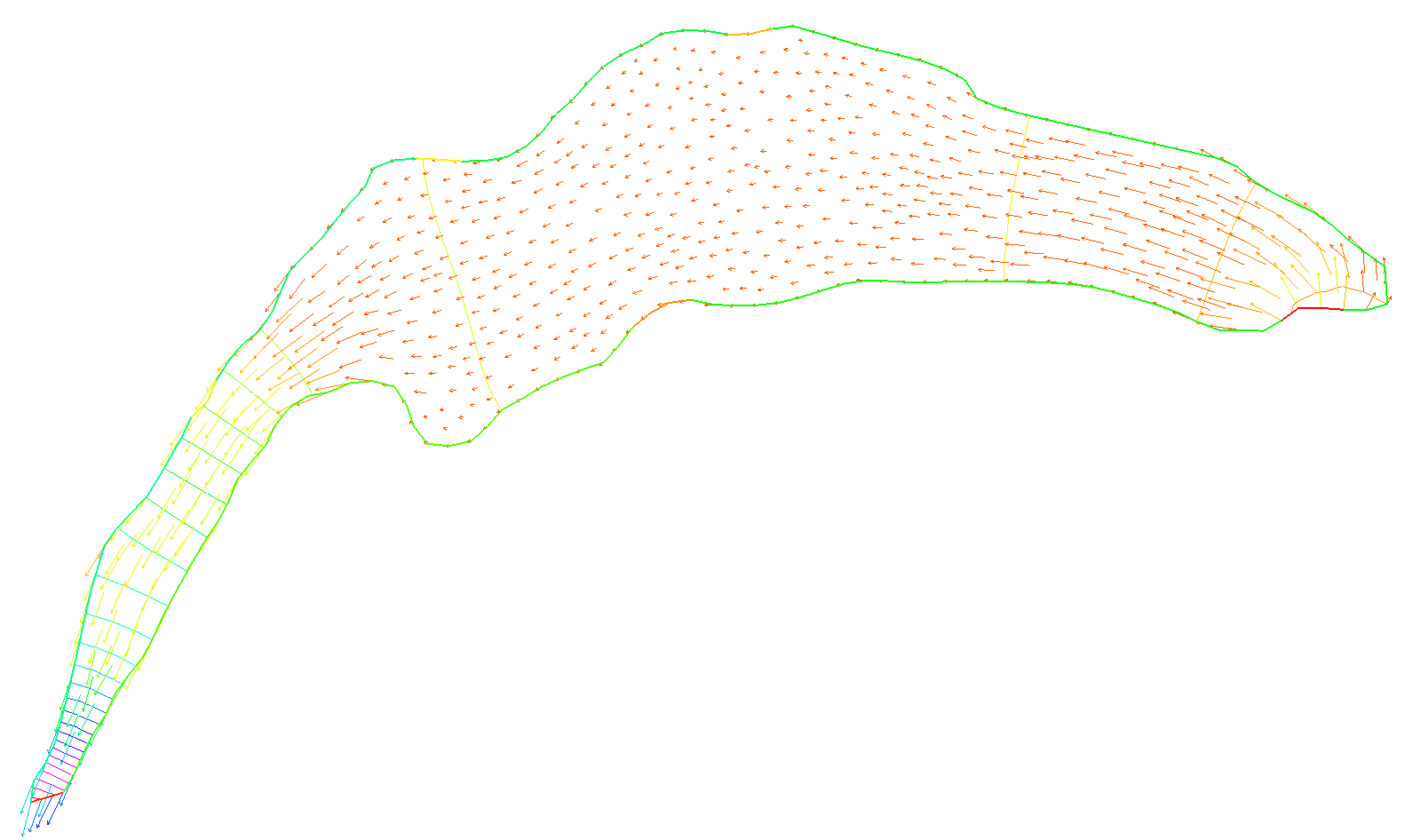}
\includegraphics[width=6cm]{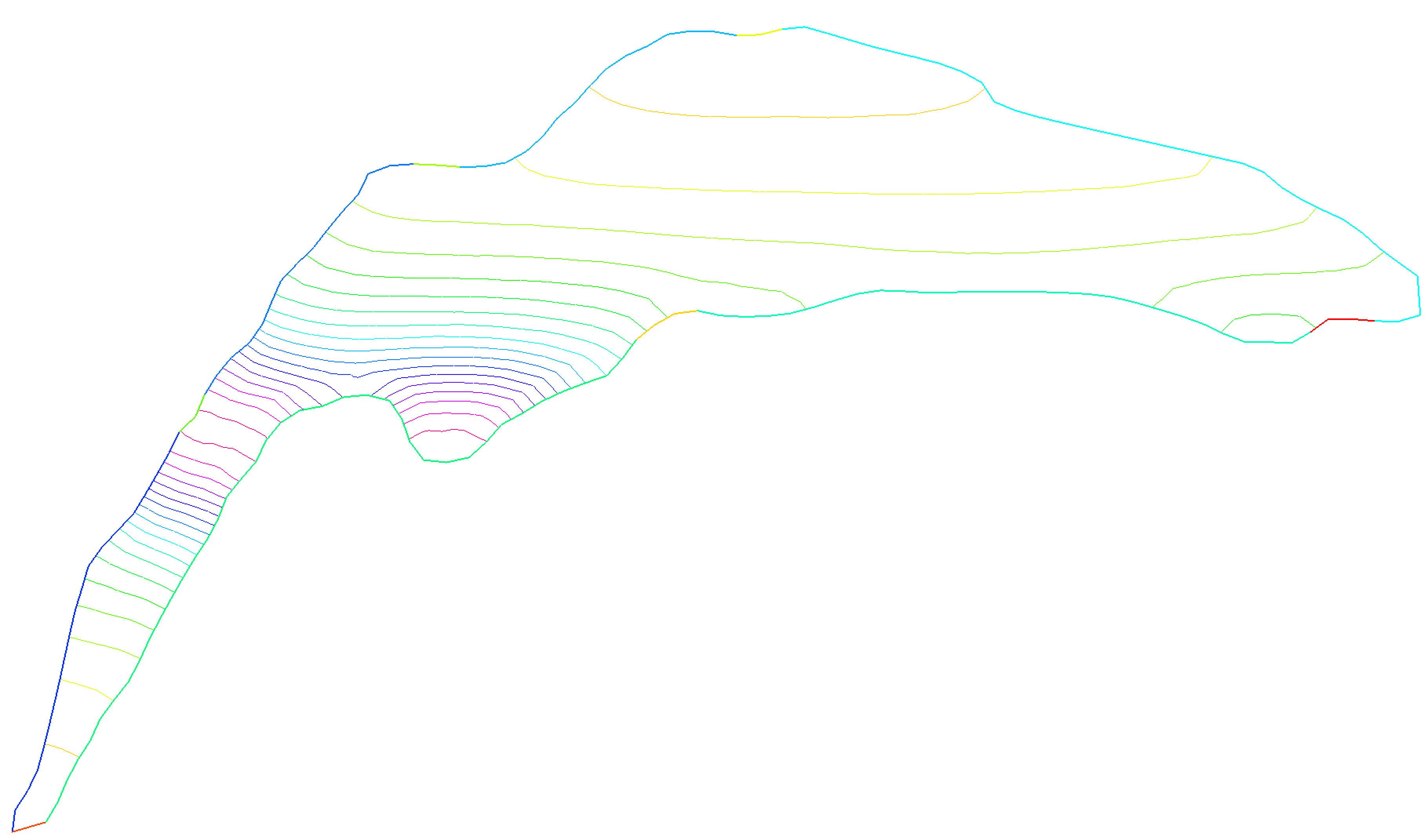}
\includegraphics[width=7cm]{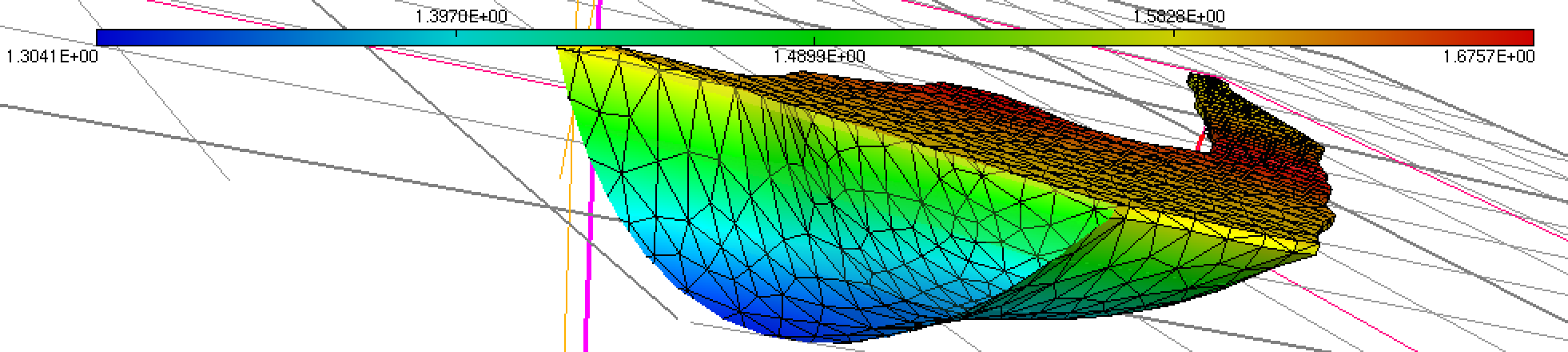}
\includegraphics[width=7cm]{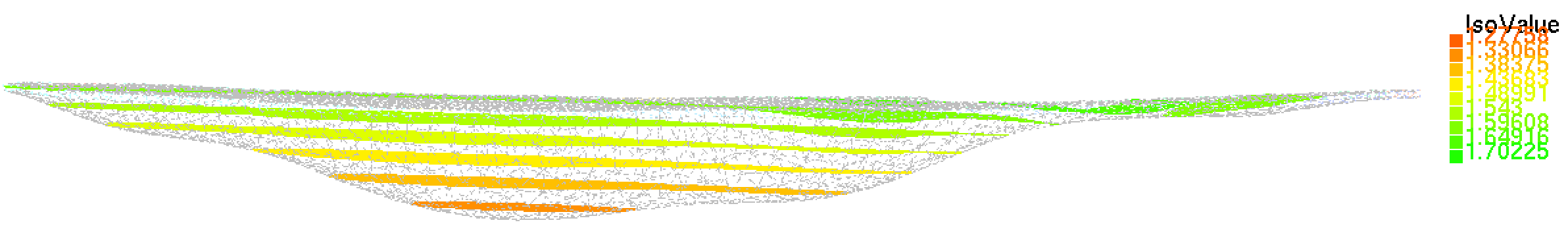}
\caption{Top left: velocity vectors and pressure isolines at the surface of the lake. 
Top right: iso lines of the surface temperature.  Bottom left: perspective view of a 3D color map of the temperature on the side of the lake past a middle vertical plane. Bottom right: perspective view showing some temperature level surfaces inside the lake. }
\label{leman}
\end{center}
\end{figure}

This computation is merely a feasibility study to prove that the implementation of the RT module in a standard CFD code is easy and fast.  Computing time on an intel core i9 takes less than a minute.

\subsection{Comments on the programming tools}~

In fifty years the research problems have become increasingly complex and without the joint development of computers and programming tools it would not be possible for a single individual to contribute or even test his ideas.  The second author is part of the team which developed the PDE solver \texttt{FreeFem++}\cite{FF} 

(see \texttt{https://fr.wikipedia.org/wiki/FreeFem\%2B\%2B}).

The algorithms discussed here have been implemented with this tool in a very short time. The discretization of lake Leman is part of the examples in \cite{FF}, written by F. Hecht.

\section{The general case, $\kappa_\nu,a_\nu$ non constant}
Photons interact with the atomic structure of the medium which implies that $\kappa_\nu$ depends strongly on $\nu$ but also on the temperature and  pressure. For the earth atmosphere the pressure and the temperature are approximately decaying exponentially with altitude.  

Assume that variations with altitude are known: $\rho\bar\kappa_\nu=\varphi(z)\kappa_\nu$ with $z=z_m=0$ on the ground.  Let $\tau=\int_0^z\varphi(s)\d s$; for instance  $\tau=1-\e^{-z}$ when $\varphi(z)=\e^{-z}$. Now (\ref{oneamu}),(\ref{onebmu}) hold with $0<\tau<Z:=1-\e^{-z_M}$ instead of $z_m<z<z_M)$.

Consider two types of scattering kernels: a Rayleigh scattering kernel $p^r(\mu,\mu')=\frac3{8}[3-\mu^2+3(\mu^2-1)\mu'^2]$ and an isotropic scattering kernel  $p=1$. Let $a^r_\nu$ and $a^i_\nu:=a_\nu-a^r_\nu$ be the scattering coefficients for both. The problem is
 \eeq{&& \label{oneamunu}
\mu\p_\tau I_\nu + \kappa_\nu I_\nu 
= \kappa_\nu(1-a_\nu) B_\nu(T)+\tfrac12\kappa_\nu \int_{-1}^1( a^r_\nu p^r +a_\nu^i) I_\nu\d\mu'
\cr&&
I(0,\mu)|_{\mu>0}= \alpha I(0,-\mu) +  Q_\nu^+(\mu),\quad I(Z,\mu)|_{\mu<0}=Q_\nu^-(\mu)
}
The boundary condition at $\tau=0$ is a simplified Lambert condition which says that a portion $\alpha$ of the incoming light is reflected back (Earth albedo) and adds to the prescribed upgoing light $Q_\nu^+$.  Sun light is prescribed at high altitude, $Z$, to be $Q^-(\mu)$.

Let 
\[
J_\nu(\tau)=\tfrac12\int_{-1}^1I_\nu(\tau,\mu)\d\mu,\quad K_\nu(\tau)=\tfrac12\int_{-1}^1\mu^2 I_\nu(\tau,\mu)\d\mu.
\] 
An integral formulation can be derived from (\ref{oneamunu}) as in \cite{CHA}, section 11.2:
\eeq{&& \label{basic}
(\mu\p_\tau + \kappa_\nu) I_\nu 
= H_\nu(\tau,\mu)
\cr&&
:=\kappa_\nu\left((1-a_\nu) B_\nu(T_\tau)+[a^i_\nu+\tfrac38a_\nu^r(3-\mu^2) ]J_\nu(\tau) + \tfrac98a_\nu^r(\mu^2-1) K_\nu(\tau)\right) ~~~~~~
\\ \label{basic2}&&\Rightarrow~~
I(\tau,\mu) = \,{\bf 1_{\mu>0}}\left[R^+_\nu(\mu) {\rm e}^{-\kappa_\nu\frac{\tau}\mu}+\int_0^\tau \frac{{\rm e}^{\kappa_\nu\frac{t-\tau}\mu}}\mu \kappa_\nu H_\nu({t,\mu})\d t\right]
\cr &&
\hskip1cm +\,{\bf 1_{\mu<0}}\left[Q^-_\nu(\mu) {\rm e}^{\kappa_\nu\frac{Z-\tau}\mu} -\int_\tau^Z \frac{{\rm e}^{\kappa_\nu\frac{t-\tau}\mu}}\mu\kappa_\nu H_\nu(t,\mu)\d t\right],
}
where $R^+(\mu)= Q_\nu^+(\mu) + \alpha I(0,-\mu)$, i.e.
\eeq{&&\label{albedo}
R^+_\nu(\mu)|_{\mu>0}=Q_\nu^+(\mu) +\alpha \left[Q^-_\nu(-\mu) {\rm e}^{-\kappa_\nu\frac{Z}\mu} +\int_0^Z \frac{{\rm e}^{-\kappa_\nu\frac{t}\mu}}\mu\kappa_\nu H_\nu(t,-\mu)\d t\right].
}
From (\ref{basic2}), since $H_\nu=H_\nu^0+\mu^2 H_\nu^2$, with $H_\nu^0,H_\nu^2$ independent of $\mu$, linear functions of $J_\nu$ and $K_\nu$:
{\small\eeq{\label{25}&&
H_\nu^0(\tau)=\kappa_\nu(1-a_\nu) B_\nu(T)+\kappa_\nu\left((a^i_\nu + \frac{9a_\nu^r}8)J_\nu - \frac{9a_\nu^r}8 K_\nu\right),
\quad \cr&&
H_\nu^2(\tau) = -\kappa_\nu\frac{3a_\nu^r}8[ J_\nu -3 K_\nu].
\\ \label{23}&
%
 &J_\nu(\tau)=\tfrac12\int_0^1\left(e^{-\kappa_\nu\frac{\tau}{\mu}}Q^+_\nu(\mu) 
+\left[e^{-\kappa_\nu\frac{(Z-\tau)}{\mu}}+\alpha e^{-\kappa_\nu\frac{(Z+\tau)}{\mu}}\right]Q^-_\nu(-\mu)\right)d\mu
\cr&&
+ \tfrac12\int_0^Z\left( [E_1(\kappa_\nu|\tau-t|)+\alpha E_1(\kappa_\nu(\tau+t))] H_\nu^0(\tau)
+ [E_3(\kappa_\nu|\tau-t|)+\alpha E_3(\kappa_\nu(\tau+t))] H_\nu^2(\tau)\right)\d t
\cr&&\\ \label{23b}&
&K_\nu(\tau) =\tfrac12\int_0^1\mu^2\left(e^{-\kappa_\nu\frac{\tau}{\mu}}Q^+_\nu(\mu)
+\left[e^{-\kappa_\nu\frac{(Z-\tau)}{\mu}}+\alpha e^{-\kappa_\nu\frac{(Z+\tau)}{\mu}}\right]Q^-_\nu(-\mu)\right)d\mu
\cr&&\cr&&
+\tfrac12  \int_0^Z\left( [E_3(\kappa_\nu|\tau-t|)+\alpha E_3(\kappa_\nu(\tau+t))] H_\nu^0(\tau)
+ [E_5(\kappa_\nu|\tau-t|)+\alpha E_5(\kappa_\nu(\tau+t))] H_\nu^2(\tau)\right)\d t,
\cr&&
}}
The system is coupled to
\eeq{&& \label{26}
\p_tT+\vu\cdot\n T -\kappa_T\Delta_{x,y,z} T + 4\pi \int_0^\infty \kappa_\nu(1-a_\nu) B_\nu(T_\tau)\d\nu=  4\pi \int_0^\infty \kappa_\nu(1-a_\nu)\ J_\nu(\tau)\d\nu ,
\cr&&
}
\subsection{Iterative method for the general case}
In the spirit of (\ref{algo}), consider 
\subsection{Algorithm 2}\label{algo2}
\begin{enumerate}
\item Starting from $T^0=0$, $J^0_\nu=0$,  $K^0_\nu=0$.
\item Compute $J_\nu^{n+1}(\tau),K_\nu^{n+1}(\tau)$ by  (\ref{23})(\ref{23b}) with  $T^n$, $J^n$,$K^n$. in place of $T,J,K$.
\item Compute $T^{n+1}$ by solving (\ref{26}) with $J^{n+1}_\nu(\tau)$ in the r.h.s.
\end{enumerate}
Note that for isotropic scattering $K_\nu$ is not needed.  Then the following convergence results hold when thermal diffusion is neglected.
\begin{theorem}\label{th:2}
Assume  $\alpha=0$, $\vu=0$, $\kappa_T=0$, $\p_tT=0$. Assume $\kappa_\nu$ strictly positive and uniformly bounded, and $0\le a_\nu<1$ for all $\nu>0$. Let $Q^\pm_\nu\ge 0$ satisfy, for some $T_M$ and some $Q$
\eeq{&&\label{condQ}
0\le Q^\pm_\nu(\mu)\le Q B_\nu(T_M) \quad \forall\mu,\nu\in(-1,1)\times\R^+.
}
Then Algorithm \ref{algo2} defines a sequence of radiative intensities $I^n_\nu$ and temperatures $T^n$ converging pointwise to $I_\nu$ and $T$ respectively, which is a solution of (\ref{oneamunu}),(\ref{26}) and the convergence is uniformely increasing.
\end{theorem}

\begin{remarks} 
~

\begin{enumerate}
\item Starting with $T^0=0$ is a sure way to initialise the recurrence and have $T^1>T^0$.
\item Most likely, monotone convergence holds also in the general case $\alpha>0$, $\vu$, $\kappa_T$ and $\p_tT$ non-zero because,just like $T\mapsto T^4$, the function $T\mapsto \int_0^\infty\kappa_\nu(1-a_\nu) B_\nu(T)\d\nu $ is monotone increasing (its derivative is strictly positive).
\item
In the special case $a_\nu^r=0$, and $Q_\nu^\pm(\mu)=|\mu|Q_\nu^\pm$ the problem is
\begin{equation}\label{RTSlabDisc}
\begin{aligned}
{}&(\mu\partial_\tau +\kappa_\nu)I_\nu(\tau,\mu)= \kappa_\nu a_\nu J_\nu(\tau)+\kappa_\nu(1-a_\nu)B_\nu(T_\tau),\qquad J_\nu(\tau) = \tfrac12 \int_{-1}^1 I_\nu(\tau,\mu)\d\mu,
\\
&I_\nu(0,\mu)=Q^+_\nu\mu\,,\quad I_\nu(Z,-\mu)=Q^-_\nu\mu\,,\qquad 0<\mu<1\,,
\\ &
\int_0^\infty\kappa_\nu(1-a_\nu) B_\nu(T_\tau)\d\nu = \int_0^\infty \kappa_\nu (1-a_\nu) J_\nu(\tau)\d \nu.
\end{aligned}
\end{equation}
The iterative process is then
to start with $T^0=0$, and compute $T^{n+1}$ from $T^n$ by
\eeq{& \label{g1}
J_\nu^{n+1}(\tau)  &= \tfrac12 Q^+_\nu E_3(\kappa_\nu\tau)+\tfrac12 Q^-_\nu E_3(\kappa_\nu(Z-\tau))
\cr&&
+ \kappa_\nu\int_0^Z E_1(\kappa_\nu|\tau-t|)\left(a_\nu J^n_\nu(t)+(1-a_\nu)B_\nu(T^n_t)\right)\d t,
\\ &&\label{g2}
\int_0^\infty\kappa_\nu(1-a_\nu) B_\nu(T^{n+1}_\tau)\d\nu = \int_0^\infty \kappa_\nu (1-a_\nu) J^{n+1}_\nu(\tau)\d \nu.
}
\item Note that $T\mapsto \int_0^\infty\kappa_\nu(1-a_\nu) B_\nu(T)\d\nu $ is  continuous,  
strictly increasing, hence invertible. Thus (\ref{g2}) defines $T^{n+1}_\tau$ uniquely.
 \item One may recover the light intensity by
\begin{equation}\label{IntForm}
\begin{aligned}
I^{n+1}_\nu(\tau,\mu)=&e^{-\kappa_\nu\frac{\tau}{\mu}}Q^+_\nu(\mu)\mathbf 1_{\mu>0}+e^{-\kappa_\nu\frac{(Z-\tau)}{|\mu|}}Q^-_\nu(\mu)\mathbf 1_{\mu<0}
\\
&+\mathbf 1_{\mu>0}\int_0^\tau e^{-\kappa_\nu\frac{(\tau-t)}{\mu}}\tfrac{\kappa_\nu}{\mu}(a_\nu J^n_\nu(t)+(1-a_\nu)B_\nu(T^n_t))dt
\\
&+\mathbf 1_{\mu<0}\int_\tau^Ze^{-\kappa_\nu\frac{(t-\tau)}{\mu}}\tfrac{\kappa_\nu}{\mu}(a_\nu J^n_\nu(t)+(1-a_\nu)B_\nu(T^n_t))dt\,.
\end{aligned}
\end{equation}
but numerically these are singular integrals while (\ref{g1}),(\ref{g2}) are not.  Indeed $\e^{-\frac{x}{\mu}}/\mu$ tends to infinity when $x$ and $\mu$ tend to $0$.
\item Theorem \ref{th:2} extends a result given in \cite{OP} which had unnecessary restrictions on $\kappa_\nu$.
\end{enumerate}
\end{remarks}
\begin{proof}~

The complete proof will appear in \cite{CBOP}. Here, for simplicity, we consider the case $a_\nu=0$.
Let $\ds S(\tau):=\int_0^\infty\frac{\kappa_\nu}2\int_0^1\left(e^{-\kappa_\nu\frac{\tau}{\mu}}Q^+_\nu(\mu)
+e^{-\kappa_\nu\frac{Z-\tau}{\mu}}Q^-_\nu(-\mu)\right)\d\mu\d\nu$.
 By (\ref{g1})
\eeqn{&&  \label{gntn2}
\ds\int_0^\infty \kappa_\nu B_\nu({T^{n+1}_\tau}) \d\nu =
\ds\int_0^\infty \kappa_\nu J^{n+1}_\nu(\tau)\d\nu  =  S(\tau)+
 \tfrac12\int_0^\infty \int_0^Z \kappa_\nu^2 E_1(\kappa_\nu|\tau-t|)B_\nu({T^n_t})  \d t\d\nu
 \cr &&
 \leq  S(\tau)+\tfrac12\max_\kappa\int_0^Z\kappa E_1(\kappa|\tau-t|)\d t\sup_{t\in(0,Z)}\int_0^\infty \kappa_\nu B_\nu({T^n_t}) \d\nu
 \cr&&
 \leq C_2 + C_1(\kappa_M)\sup_{t\in(0,Z)}\int_0^\infty \kappa_\nu B_\nu({T^n_t}) \d\nu,
}
with $C_2=\sup_{t\in(0,Z)}S(t)$ and $\kappa_M=\sup_\nu\kappa_\nu$, because $\kappa\mapsto C_1(\kappa)$ is monotone increasing. 
As $C_1(\kappa_M)<1$ it implies that $B_\nu^n(\tau):=B_\nu(T^n(\tau))$ is bounded for all $\tau$.

Now assume that $T^n_\tau>T^{n-1}_\tau$ for all $\tau>0$. Then $T\mapsto B_\nu(T)$   being increasing, $B_\nu(T^{n}_\tau)>B_\nu(T^{n-1}_\tau),~\forall\tau,\nu,$ and so for all $\tau$:
\eeq{&&\ds
\int_0^\infty\kappa_\nu\left(B_\nu(T^{n+1}_\tau)-B_\nu(T^{n}_\tau)\right)\d\nu
= \int_0^\infty\kappa_\nu\left(J_\nu^{n+1}(\tau)-J_\nu^{n}(\tau)\right)\d\nu
\cr &&
=\int_0^\infty\frac{\kappa_\nu^2}2 \int_0^Z E_1(\kappa_\nu|\tau-t|)\left(B_\nu(T^{n}_t)-B_\nu(T^{n-1}_t)\right)\d t\d\nu
>0.
\cr&&
}
As $T\mapsto B_\nu(T)$ is continuous, it implies that $T^{n+1}_\tau> T^n_\tau,~\forall\tau$. Hence for some $T^*(\tau)$, possibly $+\infty$, $T^n\to T^*$. 
By continuity $B_\nu(T^n_t)\to B_\nu(T^*_t)$, but it has been show above that $B_\nu(T^n_t)=B^n_\nu\to B^*_\nu$, so $B_\nu(T^*_t)$ is finite and so is $T^*_t$.
Recall that a bounded  increasing sequence converges, so $B_\nu(T^n_t)\to B_\nu(T^*_t)$ for all $t$ and $\nu$ and the convergence of $E_1(\kappa_\nu|\tau-t|)B_\nu(T^n_t)\to E_1(\kappa_\nu|\tau-t|)B_\nu(T^*_t)$ being monotone, the integral converges to the integral of the limit (Beppo Levi's lemma).  This shows that $T^*_\tau$ is the solution of the problem. 
\end{proof}

\section{Uniqueness, Maximum Principle}
This section follows computations in \cite{G87} (in the case $Z=+\infty$ and with $a_\nu=0$) and in \cite{MER}.

\begin{theorem}\label{T-Uniq}
Assume  $0<\kappa_\nu\le\kappa_M$,  $0\le a_\nu<1$ for all $\nu>0$. Let $Q^\pm,R^\pm\in L^1((0,1)\times\R^+)$ satisfy 
$$
0\le Q^\pm_\nu(\mu)\le R^\pm_\nu(\mu)\quad\text{ for a.e. }(\mu,\nu)\in(0,1)\times(0,\infty)\,.
$$

Then, the solutions $(I_\nu,T)$  and $(I'_\nu,T')$ of \eqref{RTSlabDisc} with $Q^\pm_\nu(\mu)$ and $R^\pm_\nu(\mu)$ respectively, satisfy
$$
I_\nu(\tau,\mu)\le I'_\nu(\tau,\mu)\text{ and }T_\tau\le T'_\tau\quad\text{ for a.e. }(\tau,\mu)\in(-1,1)\times(0,\infty)\,.
$$
In particular, 
$Q^\pm_\nu(\mu)=R^\pm_\nu(\mu)\text{ for a.e. }(\mu,\nu)\in(0,1)\times(0,\infty)$
implies 
$$
I_\nu(\tau,\mu)=I'_\nu(\tau,\mu)\text{ and }T_\tau=T'_\tau\quad\text{ for a.e. }(\tau,\mu)\in(-1,1)\times(0,\infty).
$$
\end{theorem}

\medskip
One has also the following form of a Maximum Principle.
\begin{Cor}
Let the hypotheses of Theorem \ref{th:2} hold. Let $Q^\pm_\nu(\mu)\le B_\nu(T_M)$ (resp. $Q^\pm_\nu(\mu)\ge B_\nu(T_m)$) for a.e. $(\mu,\nu)\in(0,1)\times(\R^+$. Then
 $a.e. (\tau,\mu)\in(-1,1)\times(0,\infty)$,
$$
\begin{aligned}
I_\nu(\tau,\mu)\le B_\nu(T_M)\text{ and }T_\tau\le T_M
\qquad 
\text{ resp. }I_\nu(\tau,\mu)\ge B_\nu(T_m)\text{ and }T_\tau\ge T_m
\end{aligned}
$$
\end{Cor}
The proof relies partially on a difficult argument due to \cite{MER}. It will be published in \cite{FGOP}.

\section{An application to the temperature in the Earth atmosphere}

A numerical test is reported on Figures \ref{transmittance} and \ref{fem3}.  It is an attempt at the simulation of the effect of an increase of \texttt{CO}$_2$ in the atmosphere.  Our purpose is only to assess that the numerical method can detect such a small change of $\kappa_\nu$.

Equation (\ref{g2}) is solved by a few steps of dichotomy followed by a few New steps. When $\kappa_\nu$ is larger than 4 some instabilities occur, probably in the exponential integrals. This point will be investigated in the future.

The physical and numerical parameters are
\begin{itemize}
\item Atmosphere thickness: 12km
\item Scaled sunlight power hitting the top of the atmosphere: $3.042\times 10^{-5}$
\item Percentage of sunlight reaching the ground unaffected: 0.99
\item Percentage reemitted (Earth albedo): 10\%.
\item Percentage of sunlight being a source at high altitude ($Q^-$): 0.1\%
\item Cloud (isotropic) scattering: 20\%. Cloud position : between 6 and 9km
\item Rayleigh scattering: 20\% above 9km
\item average absorption coefficient $\kappa_0= 1.225$
\item density drop versus altitude : $\rho_0 exp(-z)$
\item Discretization: 60 altitude stations, 300 frequencies (unevenly distributed)
\item Number of iterations 22. Computing time 30" per cases.
\end{itemize}
The results are very sensitive to the value of $Q^-$ and th the earth albedo.
The values for $\kappa_\nu$ are taken from Russian measurements posted on wikipedia

\texttt{https://commons.wikimedia.org/wiki/File:Atmosfaerisk\_spredning-ru.svg}.

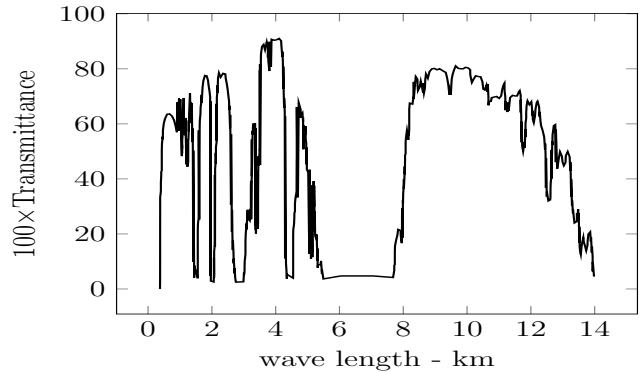
\begin{figure}[htbp]
\begin{center}
\begin{tikzpicture}[yscale=0.7]
\begin{axis}[legend style={at={(1,1)},anchor=north east}, compat=1.3,
  xlabel= {wave length - km},
  ylabel= {$100\times$Transmittance}]
\addplot[thick,solid,color=black,mark=none,mark size=1pt] table [x index=0, y index=1]{transmit1.txt};
%
\end{axis}
\end{tikzpicture}
\caption{Transmittance $t_e$ versus wave-length digitilized from 
\\
\texttt{commons.wikimedia.org/wiki/File:Atmosfaerisk\_spredning-ru.svg} 
The window around 3 is blocked by \texttt{CO}$_2$. The absorption is related to the transmittance $t_e$ by $\kappa_\nu = -\log{t_e}$.  }
\label{transmittance}
\end{center}
\end{figure}

\begin{figure}[htbp]
\begin{minipage} [b]{0.4\textwidth} 
\begin{center}
\begin{tikzpicture}[scale=0.7]
\begin{axis}[legend style={at={(1,1)},anchor=north east}, compat=1.3,
  xlabel= {Altitude - km},
  ylabel= {Scaled Temperature (K)}]
\addplot[thick,dotted,color=black,mark=none,mark size=1pt] table [x index=0, y index=1]{temperatures.txt};
\addlegendentry{ $T_{|\kappa_0}$}
\addplot[thick,solid,color=blue,mark=none,mark size=1pt] table [x index=0, y index=2]{temperaturesb.txt};
\addlegendentry{ $T_{|\kappa^1_\nu}$}
\addplot[thick,solid,color=red,mark=none,mark size=1pt] table [x index=0, y index=2]{temperatures.txt};
\addlegendentry{ $T_{|\kappa^2_\nu}$}
%
\end{axis}
\end{tikzpicture}
\label{fembig2}
\end{center}
\end{minipage}
\hskip0.5cm
\begin{minipage} [b]{0.4\textwidth}. 
\begin{center}
\begin{tikzpicture}[scale=0.7]
\begin{axis}[legend style={at={(1,1)},anchor=north east}, compat=1.3,
   xmax=2,
   xlabel= {Infrared Frequencies},
  ylabel= {Scaled $\kappa_\nu$ / Scaled mean light intensity}
  ]
\addplot[thick,solid,color=red,mark=none,mark size=1pt] table [x index=0, y index=1]{imeanz.txt};
\addlegendentry{ $J(Z)|_{\kappa^2_\nu}$}
\addplot[thick,solid,color=blue,mark=none,mark size=1pt] table [x index=0, y index=1]{imeanzb.txt};
\addlegendentry{ $J(Z)|_{\kappa^1_\nu}$}
\addplot[thick,dotted,color=black,mark=none, mark size=1pt] table [x index=0, y index=2]{imeanz.txt};
\addlegendentry{ $J(Z)|_{\kappa_0}$}

\addplot[thick,solid,color=green,mark=none, mark size=1pt] table [x index=0, y index=3]{imeanz.txt};
\addlegendentry{ $4\times\kappa^2_\nu$}

\addplot[thick,solid,color=pink,mark=none, mark size=1pt] table [x index=0, y index=3]{imeanzb.txt};
\addlegendentry{ $4\times\kappa^1_\nu$}

\end{axis}
\end{tikzpicture}
\end{center}
\end{minipage}

\caption{\label{fem3}  Scaled temperatures (left) : 3 curves $z\to T(z)$ are plotted. One computed with $\kappa_0=1.225$ which corresponds to a grey atmosphere. One with $\kappa_\nu$ shown on the right in pink color which corresponds to Figure \ref{transmittance}. The third one is with $\kappa_\nu$ shown in green on the right where the the transparent window around frequency $1$ has been blocked.  On the right the mean light intensity at altitude $Z$ are shown (mostly outgoing waves).  Filling the transparent window results in an elevation of temperature.}
\end{figure}
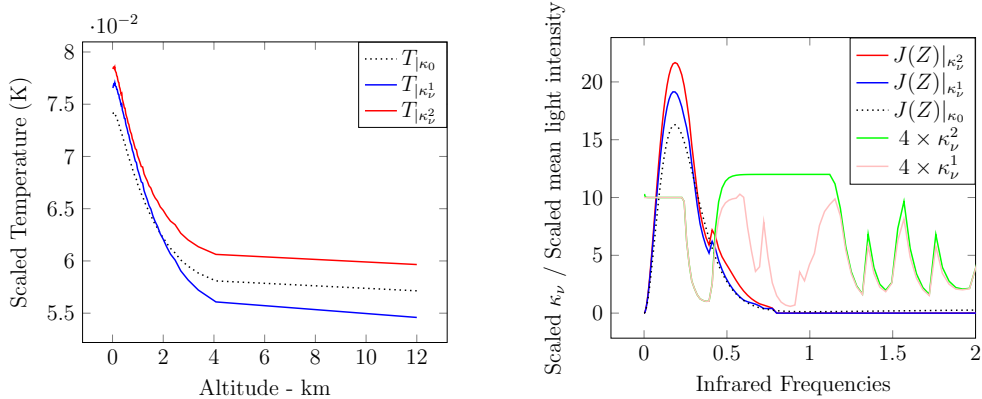

Figure transmittance https://commons.wikimedia.org/wiki/File:Atmosfaerisk\_spredning-ru.svg

\section{Conclusion}
Results obtained here are in continuation of \cite{G87},\cite{MER},\cite{BAR}, recently reviewed for possible applications to climatology in \cite{CBOP} and \cite{OP}.
Existence and uniqueness for the radiative transfer equations had remained open in the context of nuclear engineering.  For incompressible fluids it is not unrealistic to assume that the dependence of the absorption coefficient $\kappa_\nu$ upon the temperature can be replaced by an explicit dependence on altitude.  
This is the key simplification by which existence, uniqueness and monotone fast and accurate numerical schemes could be found. Hence, adding RT to a Navier-Stokes solver is easy and fast when radiations come from one direction only.

As a final remark note that it seems doable to extend the method to the general case where $\kappa_\nu$ depends on $\tau$ and $T$.  Indeed
if the dependency $\tau\mapsto\kappa_\nu(\tau)$ is guessed only approximately, then knowing $\kappa_\nu^M>\kappa_\nu(\tau)$ independent of $\tau$ is enough to apply the method with $\kappa_M$ on the left of the equation for $I_\nu$ with a correction on the right equal to $(\kappa_\nu^M-\kappa_\nu(\tau)) I_\nu$; this correction seems compatible with the monotone convergence of the temperature.  Then the method could also be extended to the case $\kappa$ function of $T$ by an additional algorithmic m-loop using $\kappa(T^m)$ instead of $\kappa(T)$ and then updating $T^m$ to the $T$ just computed.

In this article the numerical computations are only given for showing the potential of the method.  Real life applications, coupling RT to the full Navier-Stokes equations requires supercomputing power and will be done later.

\bibliographystyle{alpha}
\nocite{*}

\bibliography{samplebib}


\section{APPENDIX not part of this Compte-Rendus de M\'ecanique: Proofs of the results quoted above}

Consider the problem
\begin{equation}\label{RTSlab}
\begin{aligned}
{}&(\mu\partial_\tau +\kappa_\nu)I_\nu(\tau,\mu)=\kappa_\nu a_\nu J_\nu(\tau)+\kappa_\nu(1-a_\nu)B_\nu(T(\tau))\,,
\\
&I_\nu(0,\mu)=Q^+_\nu(\mu)\,,\quad I_\nu(Z,-\mu)=Q^-_\nu(\mu)\,,\qquad 0<\mu<1\,,
\\
&\partial_\tau\int_0^\infty\int_{-1}^1\mu I_\nu(\tau,\mu)d\mu d\nu=0\,,
\end{aligned}
\end{equation}
with the notation
$$
 J_\nu(\tau):=\tfrac12\int_{-1}^1I_\nu(\tau,\mu)d\mu\,.
$$
The last equality in \eqref{RTSlab} implies that
\begin{equation}\label{IT}
\int_0^\infty\kappa_\nu(1-a_\nu) J_\nu(\tau)d\nu=\int_0^\infty\kappa_\nu(1-a_\nu)B_\nu(T(\tau))d\nu
\end{equation}
and, assuming that $0<\kappa_\nu\le\kappa_M$ while $0\le a_\nu<1$ for all $\nu>0$, the r.h.s. of \eqref{IT} defines $T$ as a functional of $ J$, henceforth denoted $T[ J]$. Thus \eqref{RTSlab} can be recast as
\begin{equation}\label{RTSlab2}
\begin{aligned}
{}&(\mu\partial_\tau +\kappa_\nu)I_\nu(\tau,\mu)=\kappa_\nu a_\nu J_\nu(\tau)+\kappa_\nu(1-a_\nu)B_\nu(T[ J](\tau))\,,
\\
&I_\nu(0,\mu)=Q^+_\nu(\mu)\,,\quad I_\nu(Z,-\mu)=Q^-_\nu(\mu)\,,\qquad 0<\mu<1\,,
\end{aligned}
\end{equation}

In order to solve numerically \eqref{RTSlab}, one uses the method of iteration on the sources. Starting from some appropriate $(I^0_\nu,T^0)$, one construct a sequence $(I^n_\nu,T^n)$ by the following prescription
\begin{equation}\label{RTSlabDisc}
\begin{aligned}
{}&(\mu\partial_\tau +\kappa_\nu)I^{n+1}_\nu(\tau,\mu)=\kappa_\nu a_\nu J^n_\nu(\tau)+\kappa_\nu(1-a_\nu)B_\nu(T^n(\tau))\,,\quad T^n=T[ J^n_\nu]
\\
&I^{n+1}_\nu(0,\mu)=Q^+_\nu(\mu)\,,\quad I^{n+1}_\nu(Z,-\mu)=Q^-_\nu(\mu)\,,\qquad 0<\mu<1\,,
\end{aligned}
\end{equation}

Applying the method of characteristics shows that
\begin{equation}\label{IntForm}
\begin{aligned}
I^{n+1}_\nu(\tau,\mu)=&e^{-\frac{\kappa_\nu\tau}{\mu}}Q^+_\nu(\mu)\mathbf 1_{\mu>0}+e^{-\frac{\kappa_\nu(Z-\tau)}{|\mu|}}Q^-_\nu(|\mu|)\mathbf 1_{\mu<0}
\\
&+\mathbf 1_{\mu>0}\int_0^\tau e^{-\frac{\kappa_\nu(\tau-t)}{\mu}}\tfrac{\kappa_\nu}{\mu}(a_\nu J^n_\nu(t)+(1-a_\nu)B_\nu(T^n(t)))dt
\\
&+\mathbf 1_{\mu<0}\int_\tau^Ze^{-\frac{\kappa_\nu(t-\tau)}{|\mu|}}\tfrac{\kappa_\nu}{|\mu|}(a_\nu J^n_\nu(t)+(1-a_\nu)B_\nu(T^n(t)))dt\,.
\end{aligned}
\end{equation}

Since $B_\nu\ge 0$, this formula shows, by a straightforward induction argument, that
$$
I^0_\nu\ge 0\,,\,\,T^0\ge 0\,,\,\,Q^\pm_\nu\ge 0\implies I^n_\nu\ge 0\,.
$$
Moreover
$$
\begin{aligned}
I^{n+1}_\nu(\tau,\mu)-I^n_\nu(\tau,\mu)=
\\
+\mathbf 1_{\mu>0}\int_0^\tau e^{-\frac{\kappa_\nu(\tau-t)}{\mu}}\tfrac{\kappa_\nu}{\mu}a_\nu( J^n_\nu(t)- J^{n-1}_\nu(t))dt
\\
+\mathbf 1_{\mu>0}\int_0^\tau e^{-\frac{\kappa_\nu(\tau-t)}{\mu}}\tfrac{\kappa_\nu}{\mu}(1-a_\nu)(B_\nu(T^n(t))-B_\nu(T^{n-1}(t)))dt
\\
+\mathbf 1_{\mu<0}\int_\tau^Ze^{-\frac{\kappa_\nu(t-\tau)}{|\mu|}}\tfrac{\kappa_\nu}{|\mu|}a_\nu( J^n_\nu(t)- J^{n-1}_\nu(t))dt
\\
+\mathbf 1_{\mu<0}\int_\tau^Ze^{-\frac{\kappa_\nu(t-\tau)}{|\mu|}}\tfrac{\kappa_\nu}{|\mu|}(1-a_\nu)(B_\nu(T^n(t))-B_\nu(T^{n-1}(t)))dt&\,.
\end{aligned}
$$
Since $B_\nu$ is nondecreasing for each $\nu>0$, formula \eqref{IT} shows that
$$
 J^n_\nu\ge J^{n-1}_\nu\implies T^n\ge T^{n-1}
$$
we conclude from the equality above that
$$
I^0_\nu=0\,,\,\,T^0=0\,,\,\,Q^\pm_\nu\ge 0\implies
\left\{\begin{aligned}{}&0\le I^1_\nu\le I^2_\nu\le\ldots\le I^n_\nu\le\ldots
\\
&0\le T^1\le T^2\le\ldots\le T^n\le\ldots
\end{aligned}\right.
$$

Since the term
$$
(a_\nu J^n_\nu(t)+(1-a_\nu)B_\nu(T^n(t)))
$$
in both integrals on the r.h.s. is independent of $\mu$, one has
$$
\begin{aligned}
 J^{n+1}_\nu(\tau)=&\tfrac12\int_0^1\left(e^{-\frac{\kappa_\nu\tau}{\mu}}Q^+_\nu(\mu)+e^{-\frac{\kappa_\nu(Z-\tau)}{\mu}}Q^-_\nu(\mu)\right)d\mu
\\
&+\int_0^\tau\left(\int_0^1e^{-\frac{\kappa_\nu(\tau-t)}{\mu}}\frac{d\mu}{2\mu}\right)\kappa_\nu(a_\nu J^n_\nu(t)+(1-a_\nu)B_\nu(T^n(t)))dt
\\
&+\int_\tau^Z\left(\int_0^1e^{-\frac{\kappa_\nu(t-\tau)}{\mu}}\tfrac{d\mu}{2\mu}\right)\kappa_\nu(a_\nu J^n_\nu(t)+(1-a_\nu)B_\nu(T^n(t)))dt\,.
\end{aligned}
$$
One changes variables in the inner integral, so that
$$
\int_0^1e^{-\frac{X}{\mu}}\tfrac{d\mu}{\mu}=\int_1^\infty\frac{e^{-Xy}}{y}dy=\int_X^\infty\frac{e^{-z}}{z}dz=:E_1(X)\,.
$$
Thus
$$
\begin{aligned}
 J^{n+1}_\nu(\tau)=&\tfrac12\int_0^1\left(e^{-\frac{\kappa_\nu\tau}{\mu}}Q^+_\nu(\mu)+e^{-\frac{\kappa_\nu(Z-\tau)}{\mu}}Q^-_\nu(\mu)\right)d\mu
\\
&+\tfrac12\int_0^ZE_1(\kappa_\nu|\tau-t|)\kappa_\nu(a_\nu J^n_\nu(t)+(1-a_\nu)B_\nu(T^n(t)))dt\,.
\end{aligned}
$$

Integrating over $[0,Z]$ in $\tau$ implies that
$$
\begin{aligned}
\int_0^Z J^{n+1}_\nu(\tau)d\tau=\tfrac12\int_0^Z\int_0^1\left(e^{-\frac{\kappa_\nu\tau}{\mu}}Q^+_\nu(\mu)+e^{-\frac{\kappa_\nu(Z-\tau)}{\mu}}Q^-_\nu(\mu)\right)d\mu d\tau
\\
+\tfrac12\int_0^Z\left(\int_0^ZE_1(\kappa_\nu|\tau-t|)\kappa_\nu d\tau\right)(a_\nu J^n_\nu(t)+(1-a_\nu)B_\nu(T^n(t)))dt
\\
\le\tfrac12\int_0^Z\int_0^1\left(e^{-\frac{\kappa_\nu\tau}{\mu}}Q^+_\nu(\mu)+e^{-\frac{\kappa_\nu(Z-\tau)}{\mu}}Q^-_\nu(\mu)\right)d\mu d\tau
\\
+\tfrac12\left(\sup_{0\le t\le Z}\int_0^ZE_1(\kappa_\nu|\tau-t|)\kappa_\nu d\tau\right)\int_0^Z(a_\nu J^n_\nu(t)+(1-a_\nu)B_\nu(T^n(t)))dt&\,.
\end{aligned}
$$
Let us estimate the quantity
$$
\sup_{0\le t\le Z}\int_0^ZE_1(\kappa_\nu|\tau-t|)\kappa_\nu d\tau=\sup_{0\le s\le\kappa_\nu Z}\int_0^{\kappa_\nu Z}E_1(|\sigma-s|)d\sigma\,.
$$
Observe that
$$
\begin{aligned}
\int_0^{\kappa_\nu Z}E_1(|\sigma-s|)d\sigma=&\int_\mathbf{R}E_1(|\sigma-s|)1_{[0,\kappa_\nu Z]}(\sigma)d\sigma
\\
=&\int_\mathbf{R}E_1(|\theta|)1_{[-s,\kappa_\nu Z-s]}(\theta)d\theta
\\
\le&\int_\mathbf{R}E_1(|\theta|)1_{[-\kappa_\nu Z/2,\kappa_\nu Z/2]}(\theta)d\theta
\\
=&2\int_0^{\kappa_\nu Z/2}E_1(\theta)d\theta
\\
\le&2\int_0^{Z\kappa_M/2}E_1(\theta)d\theta=:2C_1\,.
\end{aligned}
$$
The first inequality is the elementary rearrangement inequality (Theorem 3.4 in \cite{LiebLoss}); the last one is based on the assumption $0<\kappa_\nu\le\kappa_M$. Thus
$$
\begin{aligned}
\int_0^Z J^{n+1}_\nu(\tau)d\tau
\le\tfrac12\int_0^Z\int_0^1\left(e^{-\frac{\kappa_\nu\tau}{\mu}}Q^+_\nu(\mu)+e^{-\frac{\kappa_\nu(Z-\tau)}{\mu}}Q^-_\nu(\mu)\right)d\mu d\tau
\\
+C_1\int_0^Z(a_\nu J^n_\nu(t)+(1-a_\nu)B_\nu(T^n(t)))dt&\,.
\end{aligned}
$$
Multiply both sides of this inequality by $\kappa_\nu$ and integrate in $\nu$: one finds that
$$
\begin{aligned}
\int_0^\infty\int_0^Z\kappa_\nu J^{n+1}_\nu(\tau)d\tau d\nu
\\
\le\tfrac12\int_0^\infty\int_0^Z\int_0^1\kappa_\nu\left(e^{-\frac{\kappa_\nu\tau}{\mu}}Q^+_\nu(\mu)+e^{-\frac{\kappa_\nu(Z-\tau)}{\mu}}Q^-_\nu(\mu)\right)d\mu d\tau d\nu
\\
+C_1\int_0^\infty\int_0^Z\kappa_\nu(a_\nu J^n_\nu(t)+(1-a_\nu)B_\nu(T^n(t)))dtd\nu&\,.
\end{aligned}
$$
At this point, we recall that $T^n=T[ J^n_\nu]$, so that
$$
\int_0^\infty\kappa_\nu(1-a_\nu)B_\nu(T^n(t)))d\nu=\int_0^\infty\kappa_\nu(1-a_\nu) J^n_\nu(t)d\nu\,,
$$
and hence
$$
\begin{aligned}
\int_0^\infty\int_0^Z\kappa_\nu J^{n+1}_\nu(\tau)d\tau d\nu\le C_1\int_0^\infty\int_0^Z\kappa_\nu J^n_\nu(t)dtd\nu
\\
+\tfrac12\int_0^\infty\int_0^Z\int_0^1\kappa_\nu\left(e^{-\frac{\kappa_\nu\tau}{\mu}}Q^+_\nu(\mu)+e^{-\frac{\kappa_\nu(Z-\tau)}{\mu}}Q^-_\nu(\mu)\right)d\mu d\tau d\nu&\,.
\end{aligned}
$$
The expression of the source term can be slightly reduced, by integrating out the $\tau$ variable:
$$
\int_0^Z\kappa_\nu e^{-\frac{\kappa_\nu\tau}{\mu}}d\tau=\int_0^Z\kappa_\nu e^{-\frac{\kappa_\nu(Z-\tau)}{\mu}}d\tau=\mu\left(1-e^{-\frac{\kappa_\nu Z}{\mu}}\right)\,,
$$
so that
$$
\begin{aligned}
0\le&\tfrac12\int_0^\infty\int_0^Z\int_0^1\kappa_\nu\left(e^{-\frac{\kappa_\nu\tau}{\mu}}Q^+_\nu(\mu)+e^{-\frac{\kappa_\nu(Z-\tau)}{\mu}}Q^-_\nu(\mu)\right)d\mu d\tau d\nu
\\
\le&\tfrac12\int_0^\infty\int_0^1(Q^+_\nu(\mu)+Q^-_\nu(\mu))\mu d\mu=:\mathcal{Q}\,.
\end{aligned}
$$
Thus
$$
\int_0^\infty\int_0^Z\kappa_\nu J^{n+1}_\nu(\tau)d\tau d\nu\le C_1\int_0^\infty\int_0^Z\kappa_\nu J^n_\nu(t)dtd\nu+\mathcal Q\,.
$$

Initializing the sequence $I^n_\nu$ with $I^0_\nu=0$ and $T^0=T[ J^0_\nu]=0$, one finds that
$$
\begin{aligned}
{}&\int_0^\infty\int_0^Z\kappa_\nu J^{1}_\nu(\tau)d\tau d\nu\le\mathcal Q\,,
\\
&\int_0^\infty\int_0^Z\kappa_\nu J^{2}_\nu(\tau)d\tau d\nu\le C_1\mathcal Q+\mathcal Q
\\
&\int_0^\infty\int_0^Z\kappa_\nu J^{3}_\nu(\tau)d\tau d\nu\le C^2_1\mathcal Q+C_1\mathcal Q+\mathcal Q
\end{aligned}
$$
and by induction
$$
\int_0^\infty\int_0^Z\kappa_\nu J^{n+1}_\nu(\tau)d\tau d\nu\le\mathcal Q\sum_{j=0}^nC_1^j\,.
$$

Since
$$
C_1=\int_0^{Z\kappa_M/2}E_1(\theta)d\theta<\int_0^{\infty}E_1(\theta)d\theta=\int_0^\infty\left(\int_1^\infty\frac{e^{-\theta y}}{y}dy\right)dy=\int_1^\infty\frac{dy}{y^2}=1\,,
$$
the series above converges and one has the uniform bound
$$
\int_0^\infty\int_0^Z\kappa_\nu J^{n+1}_\nu(\tau)d\tau d\nu\le\frac{\mathcal Q}{1-C_1}\,.
$$
Since 
$$
0\le I^1_\nu\le I^2_\nu\le\ldots\le I^n_\nu\le I^{n+1}_\nu\le\ldots
$$
the bound above and the Monotone Convergence Theorem implies that the sequence $I^{n+1}_\nu(\tau,\mu)$ converges for a.e. $(\tau,\mu,\nu)\in(0,Z)\times(-1,1)\times(0,+\infty)$ to a limit denoted $I_\nu(\tau,\mu)$ as $n\to\infty$. Since
$$
0\le T^1\le T^2\le\ldots\le T^n\le T^{n+1}\le\ldots
$$
we conclude from \eqref{IT} and the Monotone Convergence Theorem that $T^{n+1}(\tau)$ converges for a.e. $\tau\in(0,Z)$ to a limit denoted $T(\tau)$ as $n\to\infty$.

Then we can pass to the limit in \eqref{IntForm} as $n\to\infty$ by monotone convergence, to find that
$$
\begin{aligned}
I_\nu(\tau,\mu)=&e^{-\frac{\kappa_\nu\tau}{\mu}}Q^+_\nu(\mu)\mathbf 1_{\mu>0}+e^{-\frac{\kappa_\nu(Z-\tau)}{|\mu|}}Q^-_\nu(|\mu|)\mathbf 1_{\mu<0}
\\
&+\mathbf 1_{\mu>0}\int_0^\tau e^{-\frac{\kappa_\nu(\tau-t)}{\mu}}\tfrac{\kappa_\nu}{\mu}(a_\nu J_\nu(t)+(1-a_\nu)B_\nu(T(t)))dt
\\
&+\mathbf 1_{\mu<0}\int_\tau^Ze^{-\frac{\kappa_\nu(t-\tau)}{|\mu|}}\tfrac{\kappa_\nu}{|\mu|}(a_\nu J_\nu(t)+(1-a_\nu)B_\nu(T(t)))dt
\end{aligned}
$$
for a.e. $(\tau,\mu,\nu)\in(0,Z)\times(-1,1)\times(0,+\infty)$. One recognizes in this equality the integral formulation of \eqref{RTSlab} or \eqref{RTSlab2}.

Besides, since we have seen that
$$
\begin{aligned}
0=I^0_\nu\le I^1_\nu\le I^2_\nu\le\ldots\le I^n_\nu\le I^{n+1}_\nu\le\ldots\le I_\nu
\\
0=\!T^0\!\le T^1\!\le\!T^2\!\le\ldots\le T^n\le T^{n+1}\le\ldots\le T
\end{aligned}
$$
so that
$$
\begin{aligned}
0\le\int_0^Z( J^{n+1}_\nu- J^n_\nu)(\tau)d\tau
\\
=\tfrac12\int_0^Z\left(\int_0^ZE_1(\kappa_\nu|\tau-t|)\kappa_\nu d\tau\right)a_\nu( J^n_\nu- J^{n-1}_\nu)(t)dt
\\
+\tfrac12\int_0^Z\left(\int_0^ZE_1(\kappa_\nu|\tau-t|)\kappa_\nu d\tau\right)(1-a_\nu)(B_\nu(T^n(t))-B_\nu(T^{n-1}(t)))dt
\\
\le C_1\int_0^Z(a_\nu( J^n_\nu- J^{n-1}_\nu)(t)+(1-a_\nu)(B_\nu(T^n(t))-B_\nu(T^{n-1}(t)))dt&\,.
\end{aligned}
$$
Using again the equality
$$
\int_0^\infty\kappa_\nu(1-a_\nu)B_\nu(T^n(t)))d\nu=\int_0^\infty\kappa_\nu(1-a_\nu) J^n_\nu(t)d\nu\,,
$$
we conclude that
$$
0\le\int_0^Z\int_0^\infty\kappa_\nu( J^{n+1}_\nu- J^n_\nu)(\tau)d\nu d\tau\le C_1\int_0^Z\int_0^\infty\kappa_\nu( J^n_\nu- J^{n-1}_\nu)(t)dt\,.
$$
Hence
$$
0\le\int_0^Z\int_0^\infty\kappa_\nu( J^{n+1}_\nu- J^n_\nu)(\tau)d\nu d\tau\le C_1^n\int_0^\infty\kappa_\nu J^1_\nu(\tau)d\nu d\tau\le C_1^n\mathcal Q\,,
$$
so that
$$
0\le\int_0^Z\int_0^\infty\kappa_\nu( J_\nu- J^n_\nu)(\tau)d\nu d\tau\le C_1^n\int_0^\infty\kappa_\nu J^1_\nu(\tau)d\nu d\tau\le\frac{C_1^n\mathcal Q}{1-C_1}\,.
$$

Summarizing, we have proved the following result.

\begin{theorem}\label{T-Conv}
Assume that $0<\kappa_\nu\le\kappa_M$, while $0\le a_\nu<1$ for all $\nu>0$. Let $Q^\pm_\nu(\mu)$ satisfy
$$
\int_0^\infty\int_0^1\mu Q^\pm_\nu(\mu)d\mu d\nu\le\mathcal Q\,.
$$
Choose $I^0_\nu=0$ and $T^0=0$, and let $I^n_\nu$ and $T^n=T[ J^n_\nu]$ be the solution of \eqref{RTSlabDisc}. Then
$$
I^n_\nu(\tau,\mu)\to I_\nu(\tau,\mu)\quad\text{ and }\quad T^n(\tau)\to T(\tau)
$$
for $(\tau,\mu,\nu)\in(0,Z)\times(-1,1)\times(0,+\infty)$ as $n\to\infty$, where $(I_\nu,T)$ is a solution of \eqref{RTSlab} or \eqref{RTSlab2}. This method converges exponentially fast, in the sense that
$$
0\le\int_0^Z\int_0^\infty\kappa_\nu( J_\nu- J^n_\nu)(\tau)d\nu d\tau\le\frac{C_1^n\mathcal Q}{1-C_1}\,,
$$
and, if $0\le a_\nu\le a_M<1$ while $0<\kappa_m\le\kappa_\nu$, one has
$$
0\le\int_0^Z\alpha(T(t)^4-T^n(t)^4)dt\le\frac{C_1^n\mathcal Q}{\kappa_m(1-a_M)(1-C_1)}\,.
$$
\end{theorem}

The last bound comes from the defining equality for the temperature in terms of the radiative intensity
$$
\begin{aligned}
\kappa_m(1-a_M)\alpha(T^4-(T^n)^4)=&\kappa_m(1-a_M)\int_0^\infty(B_\nu(T)-B_\nu(T^n))d\nu
\\
\le&\int_0^\infty\kappa_\nu(1-a_\nu)(B_\nu(T)-B_\nu(T^n))d\nu
\\
=&\int_0^\infty\kappa_\nu(1-a_\nu)( J_\nu- J^n_\nu)d\nu\le\int_0^\infty\kappa_\nu( J_\nu- J^n_\nu)d\nu\,.
\end{aligned}
$$

\section{Uniqueness, Maximum Principle}

This section follows computations in \cite{G87} (in the case $Z=+\infty$ and with $a_\nu=0$) and in \cite{MER}.

The rather subtle monotonicity structure of the radiative transfer equations is a striking result, discovered by Mercier in \cite{MER}. In view of the complexity of the computations in \cite{MER}, it may be useful to keep in mind the following 
simple remarks, which should be viewed as a motivation.

Consider the steady radiative transfer equation \eqref{RTSlab2} without scattering ($a_\nu=0$) in the whole space with a source term $0\le S_\nu\in L^1(\mathbf R\times(-1,1)\times(0,\infty))$:
$$
\lambda I_\nu(\tau,\mu)+\mu\partial_\tau I_\nu(\tau,\mu)+\kappa_\nu I_\nu(\tau,\mu)=\kappa_\nu B_\nu(T[I])+\lambda S_\nu(\tau,\mu)\,,\quad\tau\in\mathbf R\,,\,\,|\mu|<1\,,
$$
where $\lambda>0$. By definition of $T[I]$, one easily checks that
$$
\int_{-\infty}^\infty\int_{-1}^1\int_0^\infty I_\nu(\tau,\mu)d\nu d\mu d\tau=\int_{-\infty}^\infty\int_{-1}^1\int_0^\infty S_\nu(\tau,\mu)d\nu d\mu d\tau\,.
$$
The radiative intensity is given in terms of the temperature $T[I]$ and the source $S_\nu$ by the explicit formula
$$
\begin{aligned}
I_\nu(\tau,\mu)=&\mathbf 1_{\mu>0}\int_{-\infty}^\tau e^{-\frac{(\lambda+\kappa_\nu)(\tau-t)}{\mu}}\frac{\kappa_\nu B_\nu(T[I](t))+\lambda S_\nu(t,\mu)}{\mu}dt
\\
&+\mathbf 1_{\mu<0}\int_\tau^\infty e^{-\frac{(\lambda+\kappa_\nu)(t-\tau)}{|\mu|}}\frac{\kappa_\nu B_\nu(T[I](t))+\lambda S_\nu(t,\mu)}{|\mu|}dt\,.
\end{aligned}
$$
Now, if one replaces the source of radiation $S_\nu$ in the right hand side of this equation with a larger source $S'_\nu\ge S_\nu$, it is natural to expect that the resulting radiation intensity $I'_\nu$ will be such that the associated temperature
$T[I']\ge T[I]$. Observe now that the function $T\mapsto B_\nu(T)$ is increasing on $(0,+\infty)$ for each $\nu>0$; the explicit formula for $I_\nu$ in terms of $S_\nu$ and $T[I]$ shows that $I'_\nu(\tau,\mu)\ge I_\nu(\tau,\mu)$. 

Of course, this argument is by no means rigorous, since it rests on the assumption that $S'_\nu\ge S_\nu\implies T[I']\ge T[I]$, which, although physically plausible, has not been proved yet. (Notice however that 
$$
I'_\nu\ge I_\nu\implies T[I']\ge T[I]
$$
by \eqref{IT}, since the Planck function $B_\nu$ is increasing for each $\nu>0$.) Thus, the map $S_\nu\mapsto I_\nu$ preserves both the integral and the order between radiation intensities. Now there is a clever characterization of order 
preserving maps on $L^1$ leaving the integral invariant, which is due to Crandall and Tartar \cite{CrandallTartar}. Roughly speaking, a map from $L^1$ to itself that preserves the integral is order preserving iff it is nonexpansive in $L^1$.
This brings in the notion of $L^1$-accretivity, which is at the heart\footnote{The Crandall-Tartar lemma appeared a few years before Mercier's paper \cite{MER} on the radiative transfer equation. I learned of both results in 1984, during
discussions in Mercier's lab, either from Mercier himself, or from Tartar. At the time of this writing, I cannot remember whether the Crandall-Tartar lemma was mentioned to me in connection with Mercier's result, or for some other reason.}
of Mercier's remarkable discovery.

Indeed, the monotonicity argument above, together with Proposition 1 of \cite{CrandallTartar} (with $C=L^1(\mathbf R\times(-1,1)\times(0,\infty))^+$, which is the set of a.e. positive elements of $L^1(\mathbf R\times(-1,1)\times(0,\infty))$) 
strongly suggest that it might be a good idea\footnote{This may be a reconstruction of a discussion with Mercier in the early 1980's, unless he found the $L^1$-accretivity structure of the radiative transfer equations by some other argument
which I fail to remember.} to study
$$
\int_{-\infty}^\infty\int_{-1}^1\int_0^\infty (I^2_\nu-I^1_\nu)_+(\tau,\mu)d\nu d\mu d\tau
$$
in terms of
$$
\int_{-\infty}^\infty\int_{-1}^1\int_0^\infty (S^2_\nu-S^1_\nu)_+(\tau,\mu)d\nu d\mu d\tau
$$
where $S^1_\nu,S^2_\nu\in C$ and $I^1_\nu, I^2_\nu$ are the solutions of the steady radiative transfer equation above with source terms $S^1_\nu$ and $S^2_\nu$ respectively. (Mercier's original argument is even more complex, because he
assumes that the opacity $\kappa_\nu$ depends on the temperature $T$, and is a decreasing function of $T$ for each $\nu>0$ while $T\mapsto\kappa_\nu(T)B_\nu(T)$ is nondecreasing; the reader can easily verify that the intuitive argument
above still applies, provided of course that our physically natural assumption that $S'_\nu\ge S_\nu\implies T[I']\ge T[I]$ remains valid in this case as well.)

\smallskip
Define $s_+(z)=1_{z\ge 0}$, and $z_+=\max(z,0)$ while $z_-=\max(-z,0)$. Thus
$$
z=z_+-z_-\,,\quad|z|=z_++z_-\,,\quad z_+=zs_+(z)\,.
$$
In accordance with the discussion above, we multiply both sides of the radiative transfer equation for two solutions $I_\nu$ and $I'_\nu$ by $s_+(I_\nu-I'_\nu)$ and integrate in all variables. This is precisely Mercier's computation (simpler
because $\kappa_\nu$ is independent of the temperature).

Denote 
$$
\langle\Phi\rangle:=\int_0^\infty\int_{-1}^1\Phi(\mu,\nu)d\mu d\nu
$$
With $T=T[I]$ and $T'=T[I']$ defined by \eqref{IT}, let us compute
$$
\begin{aligned}
D:=\langle\kappa_\nu((I_\nu-I'_\nu)-a_\nu( J_\nu- J'_\nu)-(1-a_\nu)(B_\nu(T)-B_\nu(T')))s_+(I_\nu-I'_\nu)\rangle
\\
=\langle\kappa_\nu(1-a_\nu)((I_\nu-I'_\nu)-(B_\nu(T)-B_\nu(T')))s_+(I_\nu-I'_\nu)\rangle
\\
+\langle\kappa_\nu a_\nu((I_\nu-I'_\nu)-( J_\nu- J'_\nu))s_+(I_\nu-I'_\nu)\rangle=:D_1+D_2
\end{aligned}
$$
Observe that
$$
\begin{aligned}
( J_\nu- J'_\nu)s_+(I_\nu(\mu)-I'_\nu(\mu))=&\tfrac12\int_{-1}^1(I_\nu-I'_\nu)(\mu')s_+(I_\nu-I'_\nu)(\mu)d\mu'
\\
\le&\tfrac12\int_{-1}^1(I_\nu-I'_\nu)_+(\mu')d\mu'\,,
\end{aligned}
$$
so that $D_2\ge 0$.

Next 
$$
D_1=\langle\kappa_\nu(1-a_\nu)((I_\nu-I'_\nu)-(B_\nu(T)-B_\nu(T')))(s_+(I_\nu-I'_\nu)-s_+(T-T'))\rangle
$$
because
$$
T=T[I]\text{ and }T'=T[I']\implies\langle\kappa_\nu(1-a_\nu)((I_\nu-I'_\nu)-(B_\nu(T)-B_\nu(T')))\rangle=0\,.
$$
Since $B_\nu$ is increasing for each $\nu>0$, one has
$$
s_+(T-T')=s_+(B_\nu(T)-B_\nu(T'))
$$
so that
$$
D_1=\langle\kappa_\nu(1-a_\nu)((I_\nu-I'_\nu)-(B_\nu(T)-B_\nu(T')))(s_+(I_\nu-I'_\nu)-s_+(B_\nu(T)-B_\nu(T')))\rangle
$$
and 
$$
s_+\text{ nondecreasing }\implies D_1\ge 0\,.
$$

Let $I_\nu$ and $J_\nu$ be two solutions of \eqref{RTSlab2} with boundary data
$$
\begin{aligned}
I_\nu(0,\mu)=Q^+_\nu(\mu)\,,\quad I_\nu(Z,-\mu)=Q^-_\nu(\mu)\,,\qquad 0<\mu<1\,,
\\
J_\nu(0,\mu)=R^+_\nu(\mu)\,,\quad J_\nu(Z,-\mu)=R^-_\nu(\mu)\,,\qquad 0<\mu<1\,.
\end{aligned}
$$
Assume that
$$
Q^\pm_\nu(\mu)\le R^\pm_\nu(\mu)\quad\text{ for a.e. }(\mu,\nu)\in(0,1)\times(0,\infty)\,.
$$
Then
$$
\begin{aligned}
\partial_\tau\langle\mu(I_\nu-J_\nu)_+\rangle\le-\langle\kappa_\nu(1-a_\nu)((I_\nu-J_\nu)-(B_\nu(T[I])-B_\nu(T[J])))s_+(I_\nu-J_\nu)\rangle
\\
-\langle\kappa_\nu a_\nu((I_\nu-J_\nu)-( J_\nu-\tilde J_\nu))s_+(I_\nu-J_\nu)\rangle\le0&\,,
\end{aligned}
$$
so that $\tau\mapsto\langle\mu(I_\nu-J_\nu)_+\rangle(\tau)$ is nonincreasing. Since
$$
\begin{aligned}
Q^-_\nu\le R^-_\nu\implies\langle\mu(I_\nu-J_\nu)_+\rangle(Z)=\langle\mu_+(I_\nu-J_\nu)_+\rangle(Z)\ge 0\,,
\\
Q^+_\nu\le R^+_\nu\implies\langle\mu(I_\nu-J_\nu)_+\rangle(0)=-\langle\mu_-(I_\nu-J_\nu)_+\rangle(0)\le 0\,,
\end{aligned}
$$
one has
$$
\begin{aligned}
\text{for a.e. }\tau\in(0,Z)\qquad0=\langle\mu(I_\nu-J_\nu)_+\rangle
\\
=\langle\kappa_\nu a_\nu((I_\nu-J_\nu)-( J_\nu-\tilde J_\nu))s_+(I_\nu-J_\nu)\rangle
\\
=\langle\kappa_\nu(1-a_\nu)((I_\nu-J_\nu)-(B_\nu(T[I])-B_\nu(T[J])))s_+(I_\nu-J_\nu)\rangle&\,,
\end{aligned}
$$
and
$$
(I_\nu-J_\nu)_+(0,-\mu)=(I_\nu-J_\nu)_+(Z,\mu)=0\qquad\text{ for a.e. }\mu\in(0,1)\,.
$$
Besides, since $\kappa_\nu(1-a_\nu)>0$ for all $\nu>0$
$$
\begin{aligned}
0=\langle\kappa_\nu(1-a_\nu)((I_\nu-J_\nu)-(B_\nu(T[I])-B_\nu(T[J])))s_+(I_\nu-J_\nu)\rangle
\\
=\langle\kappa_\nu(1-a_\nu)((I_\nu-J_\nu)-(B_\nu(T[I])-B_\nu(T[J])))(s_+(I_\nu-J_\nu)-s_+(T[I]-T[J]))\rangle
\\
\implies s_+(I_\nu(\tau,\mu)-J_\nu(\tau,\mu))=s_+(T[I]-T[J])\text{ for a.e. }(\tau,\mu,\nu)&\,.
\end{aligned}
$$

\smallskip
At this point, we must appeal to an additional idea, which is not present in Mercier's paper \cite{MER}. Since we are dealing with solutions of the radiative transfer equation having the slab symmetry, it is natural idea to use the $K$-invariant
(in the terminology of section 10 in chapter I of Chandrasekhar \cite{CHA}). This idea\footnote{A somewhat similar idea, unfortunately unpublished, had been used by R. Sentis to simplify the uniqueness proof for the linear Milne problem studied
in \cite{BSS}.} is at the heart of the exponential decay estimate for the Milne problem obtained in \cite{G87}, and will be used here for a different purpose.

We compute
$$
\begin{aligned}
\partial_\tau\left\langle\frac{\mu^2}{\kappa_\nu}(I_\nu-J_\nu)_+\right\rangle=-\langle a_\nu\mu((I_\nu-J_\nu)-( J_\nu-\tilde J_\nu))s_+(T[I]-T[J])\rangle
\\
-\langle (1-a_\nu)\mu((I_\nu-J_\nu)-(B_\nu(T[I])-B_\nu(T[J]))s_+(T[I]-T[J])\rangle
\\
=-\langle a_\nu\mu(I_\nu-J_\nu)s_+(T[I]-T[J])\rangle-\langle (1-a_\nu)\mu(I_\nu-J_\nu)s_+(T[I]-T[J])\rangle
\\
=-\langle\mu(I_\nu-J_\nu)s_+(T[I]-T[J])\rangle=-\langle\mu(I_\nu-J_\nu)_+\rangle=0&\,,
\end{aligned}
$$
since
$$
\int_{-1}^1\mu( J_\nu(\tau)-\tilde J_\nu(\tau))d\mu=\int_{-1}^1\mu(B_\nu(T[I])-B_\nu(T[J]))d\mu=0\,.
$$
Next we integrate in $\tau\in(0,Z)$, and observe that
$$
\begin{aligned}
(I_\nu-J_\nu)_+(0,-\mu)=0\text{ and }Q^+_\nu(\mu)\le R^+_\nu(\mu)\qquad\text{ for a.e. }\mu\in(0,1)
\\
\implies\left\langle\frac{\mu^2}{\kappa_\nu}(I_\nu-J_\nu)_+\right\rangle(\tau)=\left\langle\frac{\mu^2}{\kappa_\nu}(I_\nu-J_\nu)_+\right\rangle(0)=0&\,.
\end{aligned}
$$

Summarizing, we have proved the following result.
\begin{theorem}\label{T-Uniq}
Assume that $0<\kappa_\nu\le\kappa_M$, while $0\le a_\nu<1$ for all $\nu>0$. Let $Q^\pm,R^\pm\in L^1((0,1)\times(0,\infty))$ satisfy 
$$
0\le Q^\pm_\nu(\mu)\le R^\pm_\nu(\mu)\quad\text{ for a.e. }(\mu,\nu)\in(0,1)\times(0,\infty)\,.
$$

Then, the solutions $(I_\nu,T[I])$ of \eqref{RTSlab2}, and $(J_\nu,T[J])$ of \eqref{RTSlab2} with boundary data $Q^\pm_\nu(\mu)$ replaced with $R^\pm_\nu(\mu)$ satisfy
$$
I_\nu(\tau,\mu)\le J_\nu(\tau,\mu)\text{ and }T[I](\tau)\le T[J](\tau)\quad\text{ for a.e. }(\tau,\mu)\in(-1,1)\times(0,\infty)\,.
$$

In particular, 
$$
\begin{aligned}
Q^\pm_\nu(\mu)=R^\pm_\nu(\mu)\text{ for a.e. }(\mu,\nu)\in(0,1)\times(0,\infty)
\\
\implies I_\nu(\tau,\mu)=J_\nu(\tau,\mu)\text{ and }T[I](\tau)=T[J](\tau)\quad\text{ for a.e. }(\tau,\mu)\in(-1,1)\times(0,\infty)&\,.
\end{aligned}
$$
\end{theorem}

\medskip
One has also the following form of Maximum Principle.
\begin{Cor}
Assume that $0<\kappa_\nu\le\kappa_M$, while $0\le a_\nu<1$ for all $\nu>0$. Let $Q^\pm_\nu(\mu)\le B_\nu(T_M)$ (resp. $Q^\pm_\nu(\mu)\ge B_\nu(T_m)$) for a.e. $(\mu,\nu)\in(0,1)\times(0,\infty)$. Then
$$
\begin{aligned}
I_\nu(\tau,\mu)\le B_\nu(T_M)\text{ and }T[I](\tau)\le T_M
\\
\text{ resp. }I_\nu(\tau,\mu)\ge B_\nu(T_m)\text{ and }T[I](\tau)\ge T_m
\\
\quad\text{ for a.e. }(\tau,\mu)\in(-1,1)\times(0,\infty)&\,.
\end{aligned}
$$
\end{Cor}

\begin{proof}
Indeed, $J_\nu=B_\nu(T_M)$ and $T[J]=T_M$ (resp. $J_\nu=B_\nu(T_m)$ and $T[J]=T_m$) is the solution of \eqref{RTSlab2} with boundary data $R^\pm_\nu=B_\nu(T_M)$ (resp. $R^\pm_\nu=B_\nu(T_m)$).
\end{proof}

\smallskip
In Theorem \ref{T-Conv}, if one has the stronger condition 
$$
0\le Q^\pm_\nu(\mu)\le B_\nu(T_M)\quad\text{ for a.e. }(\mu,\nu)\in(0,1)\times(0,\infty)
$$
one obtains the following bound for the numerical and theoretical solutions
$$
0\le I^1_\nu\le I^2_\nu\le\ldots\le I^n_\nu\le\ldots I_\nu\le B_\nu(T_M)
$$
while
$$
0\le T^1\le T_2\le\ldots\le T^n\le\ldots\le T\le T_M\,.
$$

\section{Radiative Transfer with Rayleigh Scattering in a Slab}

In this section, we discuss the same problem as in the previous section, with the isotropic scattering replaced by the Rayleigh phase function. In the case of slab symmetry, the Rayleigh phase function is
$$
p(\mu,\mu')=\tfrac3{16}(3-\mu^2)+\tfrac3{16}(3\mu^2-1)\mu'^2
$$
(see section 11.2 in chapter I of \cite{CHA}). Observe that
\begin{equation}\label{p>0}
p(\mu,\mu')=\tfrac3{16}(3+3\mu^2\mu'^2-\mu^2-\mu'^2)\ge\tfrac3{16}>0\,,
\end{equation}
while
\begin{equation}\label{intp=1}
\tfrac12\int_{-1}^1p(\mu,\mu')d\mu=\tfrac3{16}(6+3\cdot\tfrac23\mu'^2-\tfrac23-2\mu'^2)=1\,.
\end{equation}
Keeping \eqref{IT} as the defining equation for $T[I]$, the problem \eqref{RTSlab2} becomes
\begin{equation}\label{RTSlab2R}
\begin{aligned}
{}&(\mu\partial_\tau +\kappa_\nu)I_\nu(\tau,\mu)=&&\!\!\!\tfrac38\kappa_\nu a_\nu((3-\mu^2) J_\nu(\tau)+(3\mu^2-1){K_\nu}(\tau))
\\
& &&\!\!\!+\kappa_\nu(1-a_\nu)B_\nu(T[ J](\tau))\,,
\\
&I_\nu(0,\mu)=Q^+_\nu(\mu)\,,\quad&& I_\nu(Z,-\mu)=Q^-_\nu(\mu)\,,\qquad 0<\mu<1\,,
\end{aligned}
\end{equation}
with $K=\tfrac12\int_{-1}^1\mu^2I\d\mu$, and one easily checks that \eqref{IT} and \eqref{intp=1} imply that
$$
\partial_\tau\int_0^\infty\int_{-1}^1\mu I_\nu(\tau,\mu)d\mu d\nu=0\,.
$$

Starting from $I^0_\nu(\tau,\mu)=0$ and $T^0(\tau)=0$, one solves for $I^{n+1}$ 
\begin{equation}\label{RTSlab2RDisc}
\begin{aligned}
{}&(\mu\partial_\tau +\kappa_\nu)I^{n+1}_\nu(\tau,\mu)=&&\!\!\!\tfrac38\kappa_\nu a_\nu((3-\mu^2) J^n_\nu(\tau)+(3\mu^2-1){K^n_\nu}(\tau))
\\
& &&\!\!\!+\kappa_\nu(1-a_\nu)B_\nu(T^n(\tau))\,,\qquad T^n:=T[ J^n]
\\
&I^{n+1}_\nu(0,\mu)=Q^+_\nu(\mu)\,,\quad&& I^{n+1}_\nu(Z,-\mu)=Q^-_\nu(\mu)\,,\qquad 0<\mu<1\,.
\end{aligned}
\end{equation}
Since $B_\nu$ is nondecreasing for each $\nu>0$, one easily checks with \eqref{p>0} that
$$
\begin{aligned}
0=I^0_\nu\le I^1_\nu\le I^2_\nu\le\ldots\le I^n_\nu\le I^{n+1}_\nu\le\ldots
\\
0=\!T^0\!\le T^1\!\le\!T^2\!\le\ldots\le T^n\le T^{n+1}\le\ldots
\end{aligned}
$$

Explicitly
\begin{equation}\label{IntFormR}
\begin{aligned}
I^{n+1}_\nu(\tau,\mu)=&e^{-\frac{\kappa_\nu\tau}{\mu}}Q^+_\nu(\mu)\mathbf 1_{\mu>0}+e^{-\frac{\kappa_\nu(Z-\tau)}{|\mu|}}Q^-_\nu(|\mu|)\mathbf 1_{\mu<0}
\\
&+\mathbf 1_{\mu>0}\int_0^\tau e^{-\frac{\kappa_\nu(\tau-t)}{\mu}}\tfrac{\kappa_\nu}{\mu}\tfrac38a_\nu((3-\mu^2) J^n_\nu(t)+(3\mu^2-1){K^n_\nu}(t))dt
\\
&+\mathbf 1_{\mu>0}\int_0^\tau e^{-\frac{\kappa_\nu(\tau-t)}{\mu}}\tfrac{\kappa_\nu}{\mu}(1-a_\nu)B_\nu(T^n(t))dt
\\
&+\mathbf 1_{\mu<0}\int_\tau^Ze^{-\frac{\kappa_\nu(t-\tau)}{|\mu|}}\tfrac{\kappa_\nu}{|\mu|}\tfrac38a_\nu((3-\mu^2) J^n_\nu(t)+(3\mu^2-1){K^n_\nu}(t))dt
\\
&+\mathbf 1_{\mu<0}\int_\tau^Ze^{-\frac{\kappa_\nu(t-\tau)}{|\mu|}}\tfrac{\kappa_\nu}{|\mu|}(1-a_\nu)B_\nu(T^n(t))dt\,.
\end{aligned}
\end{equation}

This scheme can be reduced to the following iteration to compute first $ J_\nu$ and $\tilde\mu^2I_\nu$ as follows:
$$
\begin{aligned}
 J^{n+1}_\nu(\tau)=&\tfrac12\int_0^1\left(e^{-\frac{\kappa_\nu\tau}{\mu}}Q^+_\nu(\mu)\mathbf 1_{\mu>0}+e^{-\frac{\kappa_\nu(Z-\tau)}{|\mu|}}Q^-_\nu(|\mu|)\mathbf 1_{\mu<0}\right)d\mu
\\
&+\tfrac3{16}\int_0^ZE_1(\kappa_\nu|\tau-t|)\kappa_\nu a_\nu(3 J^n_\nu(t)-{K^n_\nu}(t))dt
\\
&+\tfrac3{16}\int_0^ZE_3(\kappa_\nu|\tau-t|)\kappa_\nu a_\nu(3{K^n_\nu}(t)- J^n_\nu(t))dt
\\
&+\tfrac12\int_0^ZE_1(\kappa_\nu|\tau-t|)\kappa_\nu(1-a_\nu)B_\nu(T^n(t))dt\,,
\end{aligned}
$$
while
$$
\begin{aligned}
{K^{n+1}_\nu}(\tau)=&\tfrac12\int_0^1\left(e^{-\frac{\kappa_\nu\tau}{\mu}}Q^+_\nu(\mu)\mathbf 1_{\mu>0}+e^{-\frac{\kappa_\nu(Z-\tau)}{|\mu|}}Q^-_\nu(|\mu|)\mathbf 1_{\mu<0}\right)\mu^2d\mu
\\
&+\tfrac3{16}\int_0^ZE_3(\kappa_\nu|\tau-t|)\kappa_\nu a_\nu(3 J^n_\nu(t)-{K^n_\nu}(t))dt
\\
&+\tfrac3{16}\int_0^ZE_5(\kappa_\nu|\tau-t|)\kappa_\nu a_\nu(3{K^n_\nu}(t)- J^n_\nu(t))dt
\\
&+\tfrac12\int_0^ZE_3(\kappa_\nu|\tau-t|)\kappa_\nu(1-a_\nu)B_\nu(T^n(t))dt\,,
\end{aligned}
$$
where we have denoted
$$
E_n(X):=\int_1^\infty\frac{e^{-Xy}}{y^n}dy=\int_X^\infty\frac{e^{-z}}{z^n}dz=\int_0^1e^{-X/\mu}\mu^{n-2}d\mu\,.
$$

Once $ J_\nu$ and ${K_\nu}$ are known, the right hand side of \eqref{RTSlab2R} is known, and $I_\nu$ is obtained by a simple quadrature formula.

\smallskip
Returning to \eqref{IntFormR}, assume that 
$$
0\le Q^\pm_\nu\le B_\nu(T_M)\,,\,\,0\le I^n_\nu\le B_\nu(T_M)\text{ and }0\le T^n\le T_M\,.
$$
Then
$$
\begin{aligned}
I^{n+1}_\nu(\tau,\mu)\le&\left(e^{-\frac{\kappa_\nu\tau}{\mu}}\mathbf 1_{\mu>0}+e^{-\frac{\kappa_\nu(Z-\tau)}{|\mu|}}\mathbf 1_{\mu<0}\right)B_\nu(T_M)
\\
&+\mathbf 1_{\mu>0}\int_0^\tau e^{-\frac{\kappa_\nu(\tau-t)}{\mu}}\tfrac{\kappa_\nu}{\mu}\tfrac38a_\nu((3-\mu^2)B_\nu(T_M)+(\mu^2-\tfrac13)B_\nu(T_M))dt
\\
&+\mathbf 1_{\mu>0}\int_0^\tau e^{-\frac{\kappa_\nu(\tau-t)}{\mu}}\tfrac{\kappa_\nu}{\mu}(1-a_\nu)B_\nu(T_M)dt
\\
&+\mathbf 1_{\mu<0}\int_\tau^Ze^{-\frac{\kappa_\nu(t-\tau)}{|\mu|}}\tfrac{\kappa_\nu}{|\mu|}\tfrac38a_\nu((3-\mu^2)B_\nu(T_M)+(\mu^2-\tfrac13)B_\nu(T_M))dt
\\
&+\mathbf 1_{\mu<0}\int_\tau^Ze^{-\frac{\kappa_\nu(t-\tau)}{|\mu|}}\tfrac{\kappa_\nu}{|\mu|}(1-a_\nu)B_\nu(T_M)dt
\\
=&\left(e^{-\frac{\kappa_\nu\tau}{\mu}}\mathbf 1_{\mu>0}+e^{-\frac{\kappa_\nu(Z-\tau)}{|\mu|}}\mathbf 1_{\mu<0}\right)B_\nu(T_M)
\\
&+\mathbf 1_{\mu>0}\int_0^\tau e^{-\frac{\kappa_\nu(\tau-t)}{\mu}}\tfrac{\kappa_\nu}{\mu}(\tfrac38a_\nu(3-\tfrac13)+(1-a_\nu))B_\nu(T_M))dt
\\
&+\mathbf 1_{\mu<0}\int_\tau^Ze^{-\frac{\kappa_\nu(t-\tau)}{|\mu|}}\tfrac{\kappa_\nu}{|\mu|}(\tfrac38a_\nu(3-\tfrac13)+(1-a_\nu))B_\nu(T_M))dt
\\
=&\left(e^{-\frac{\kappa_\nu\tau}{\mu}}\mathbf 1_{\mu>0}+e^{-\frac{\kappa_\nu(Z-\tau)}{|\mu|}}\mathbf 1_{\mu<0}\right)B_\nu(T_M)
\\
&+B_\nu(T_M)\left(\mathbf 1_{\mu>0}\int_0^\tau e^{-\frac{\kappa_\nu(\tau-t)}{\mu}}\tfrac{\kappa_\nu}{\mu}dt+\mathbf 1_{\mu<0}\int_\tau^Ze^{-\frac{\kappa_\nu(t-\tau)}{|\mu|}}\tfrac{\kappa_\nu}{|\mu|}dt\right)
\\
=&\left(e^{-\frac{\kappa_\nu\tau}{\mu}}\mathbf 1_{\mu>0}+e^{-\frac{\kappa_\nu(Z-\tau)}{|\mu|}}\mathbf 1_{\mu<0}\right)B_\nu(T_M)
\\
&+B_\nu(T_M)\left(\mathbf 1_{\mu>0}\left(1-e^{-\frac{\kappa_\nu\tau}{\mu}}\right)+\mathbf 1_{\mu<0}\left(1-e^{-\frac{\kappa_\nu(Z-\tau)}{|\mu|}}\right)\right)
\\
=&B_\nu(T_M)\,.
\end{aligned}
$$
Besides
$$
T^{n+1}=T[I^{n+1}]\le T[B_\nu(T_M)]=T_M
$$
(using again that $T\mapsto B_\nu(T)$ is increasing for each $\nu>0$ while $\kappa_\nu(1-a_\nu)>0$ for all $\nu>0$).

\smallskip
Summarizing, we have proved the following result.

\begin{theorem}\label{T-CvgR}
Assume that $\kappa_\nu>0$ while $0\le a_\nu<1$ for all $\nu>0$. Let $Q^\pm$ satisfy
$$
0\le Q^\pm_\nu(\mu)\le B_\nu(T_M)\qquad\text{ for all }\mu\in(-1,1)\text{ and }\nu>0\,.
$$
The iteration method \eqref{RTSlab2RDisc} starting from $I^0_\nu=0$ and $T^0=0$ defines a sequence of radiative intensities $I^n_\nu$ and temperatures $T^n$ converging pointwise to $I_\nu$ and $T=T[I]$ respectively, which is a solution of \eqref{RTSlab2R}.
\end{theorem}

The argument above is based on the monotonicity of the sequences $I^n_\nu$ and $T^n$, and does not give any information on the convergence rate.

Finally, Theorem \ref{T-Uniq} holds verbatim for the problem \eqref{RTSlab2R}. Here are the (slight) modifications to the proof due to the Rayleigh phase function.

First, we slightly modify the argument concerning the term $D_2$ as follows. In the case of the Rayleigh phase function
$$
\begin{aligned}
D_2=\tfrac12\int_0^\infty\kappa_\nu a_\nu\int_{-1}^1(I_\nu-I'_\nu)_+(\mu)d\mu d\nu
\\
-\tfrac12\int_0^\infty\kappa_\nu a_\nu\int_{-1}^1\int_{-1}^1p(\mu,\mu')(I_\nu- J'_\nu)(\mu')s_+(I_\nu-I'_\nu)(\mu)d\mu' d\mu d\nu&\,.
\end{aligned}
$$
Since $p\ge 0$, one has
$$
p(\mu,\mu')(I_\nu- J'_\nu)(\mu')s_+(I_\nu-I'_\nu)(\mu)\le p(\mu,\mu')(I_\nu- J'_\nu)_+(\mu')
$$
so that
$$
\begin{aligned}
D_2\ge\tfrac12\int_0^\infty\kappa_\nu a_\nu\int_{-1}^1(I_\nu-I'_\nu)_+(\mu)d\mu d\nu
\\
-\tfrac12\int_0^\infty\kappa_\nu a_\nu\int_{-1}^1\int_{-1}^1p(\mu,\mu')(I_\nu- J'_\nu)_+(\mu')d\mu' d\mu d\nu=0
\end{aligned}
$$
since
$$
\tfrac12\int_{-1}^1p(\mu,\mu')d\mu=1\,.
$$
Therefore, following the proof of Theorem \ref{T-Uniq}, we obtain in the same manner the following conclusions
$$
\langle\mu(I_\nu-J_\nu)_+\rangle(\tau)=0\text{ for a.e. }\tau\in(0,Z)\,,
$$
and
$$
\begin{aligned}
s_+(I_\nu(\tau,\mu)-J_\nu(\tau,\mu))=s_+(T[I](\tau)-T[J](\tau))
\\
\text{ for a.e. }(\tau,\mu,\nu)\in(0,Z)\times(-1,1)\times(0,\infty)&\,,
\end{aligned}
$$
while
$$
(I_\nu-J_\nu)_+(0,-\mu)=(I_\nu-J_\nu)_+(Z,\mu)=0\qquad\text{ for a.e. }\mu\in(0,1)\,.
$$

Next we compute
$$
\begin{aligned}
\partial_\tau\left\langle\frac{\mu^2}{\kappa_\nu}(I_\nu-J_\nu)_+\right\rangle=-\tfrac12\int_0^\infty a_\nu\int_{-1}^1\mu(I_\nu-J_\nu)_+(\tau,\mu)d\mu d\nu 
\\
+\tfrac12\int_0^\infty a_\nu\int_{-1}^1\mu\int_{-1}^1p(\mu,\mu')(I_\nu-J_\nu)_+(\tau,\mu')d\mu' d\nu\,s_+(T[I](\tau)-T[J](\tau))
\\
-\langle (1-a_\nu)\mu((I_\nu-J_\nu)-(B_\nu(T[I])-B_\nu(T[J]))s_+(T[I]-T[J])\rangle
\\
=-\langle a_\nu\mu(I_\nu-J_\nu)s_+(T[I]-T[J])\rangle-\langle (1-a_\nu)\mu(I_\nu-J_\nu)s_+(T[I]-T[J])\rangle
\\
=-\langle\mu(I_\nu-J_\nu)s_+(T[I]-T[J])\rangle=-\langle\mu(I_\nu-J_\nu)_+\rangle=0&\,,
\end{aligned}
$$
since
$$
\int_{-1}^1\mu p(\mu,\mu')d\mu=\int_{-1}^1\mu(B_\nu(T[I])-B_\nu(T[J]))d\mu=0\,.
$$
Finally we integrate in $\tau\in(0,Z)$, and conclude as in the previous section that
$$
\begin{aligned}
(I_\nu-J_\nu)_+(0,-\mu)=0\text{ and }Q^+_\nu(\mu)\le R^+_\nu(\mu)\qquad\text{ for a.e. }\mu\in(0,1)
\\
\implies\left\langle\frac{\mu^2}{\kappa_\nu}(I_\nu-J_\nu)_+\right\rangle(\tau)=\left\langle\frac{\mu^2}{\kappa_\nu}(I_\nu-J_\nu)_+\right\rangle(0)=0&\,.
\end{aligned}
$$

\end{document}